\documentclass[a4paper,reqno]{amsart}
\usepackage{amssymb, amsfonts, amsthm,amsmath,amsaddr}
\usepackage{tikz}
\usepackage{bbm}
\usepackage{enumerate}
\usepackage{enumitem}
\usepackage{fullpage}
\usetikzlibrary{calc}
\allowdisplaybreaks

\usepackage{verbatim}
\usepackage[linesnumbered,ruled,vlined]{algorithm2e}

\newtheorem{theorem}{Theorem}
\newtheorem{lemma}[theorem]{Lemma}

\newtheorem{corollary}[theorem]{Corollary}
\newtheorem{claim}[theorem]{Claim}

\newtheorem{obs}[theorem]{Observation}
\newtheorem{definition}[theorem]{Definition}

\newcommand{\eps}{\varepsilon}

\newcommand{\cond}{\left|\right.}

\newcommand{\timescale}{T}

\newcommand{\infec}{A}
\newcommand{\ninfecrat}{\alpha}
\newcommand{\linfec}{B}
\newcommand{\nlinfecappr}{b}
\newcommand{\nlinfecrat}{\beta}

\newcommand{\uinfec}{C}
\newcommand{\nuinfecappr}{c}
\newcommand{\nuinfecrat}{\gamma}

\newcommand{\infedge}{F}
\newcommand{\dom}{\mathcal{S}}

\newcommand{\cross}{\mathcal{N}}

\newcommand{\overlap}{\mathcal{W}}

\newcommand{\multi}{\mathcal{H}}

\newcommand{\coll}{\mathcal{C}}
\newcommand{\kset}{K}

\newcommand{\bpo}{M}
\newcommand{\bps}{M'}
\newcommand{\bpd}{Z}
\newcommand{\act}{\widehat{C}}

\newcommand{\convexpath}[2]{
   [   
   create hullcoords/.code={
     \global\edef\namelist{#1}
     \foreach [count=\counter] \nodename in \namelist {
       \global\edef\numberofnodes{\counter}
       \coordinate (hullcoord\counter) at (\nodename);
     }
     \coordinate (hullcoord0) at (hullcoord\numberofnodes);
     \pgfmathtruncatemacro\lastnumber{\numberofnodes+1}
     \coordinate (hullcoord\lastnumber) at (hullcoord1);
   },
   create hullcoords
   ]
   ($(hullcoord1)!#2!-90:(hullcoord0)$)
   \foreach [
   evaluate=\currentnode as \previousnode using \currentnode-1,
   evaluate=\currentnode as \nextnode using \currentnode+1
   ] \currentnode in {1,...,\numberofnodes} {
     let \p1 = ($(hullcoord\currentnode) - (hullcoord\previousnode)$),
     \n1 = {atan2(\y1,\x1) + 90},
     \p2 = ($(hullcoord\nextnode) - (hullcoord\currentnode)$),
     \n2 = {atan2(\y2,\x2) + 90},
     \n{delta} = {Mod(\n2-\n1,360) - 360}
     in 
     {arc [start angle=\n1, delta angle=\n{delta}, radius=#2]}
     -- ($(hullcoord\nextnode)!#2!-90:(hullcoord\currentnode)$) 
   }
 }

\title{Bootstrap Percolation on the Binomial Random $k$-uniform Hypergraph}

\author{Mihyun Kang}
\email{kang@math.tugraz.at}
\address{Institute of Discrete Mathematics\\ Graz University of Technology\\8010 Graz, Austria}

\author{Christoph Koch}	
\email{ckoch@math.tugraz.at}
\address{Institute of Discrete Mathematics\\ Graz University of Technology\\8010 Graz, Austria}	

\author{Tam\'as Makai}
\email{makai@math.lmu.de}
\address{Department of Mathematics\\ LMU Munich\\ 80333 Munich, Germany}

\date{}

\begin{document}
\begin{abstract}
We investigate the behaviour of $r$-neighbourhood bootstrap percolation  on the binomial $k$-uniform random hypergraph $H_k(n,p)$ for given integers $k\geq 2$ and $r\geq 2$. In $r$-neighbourhood bootstrap percolation, infection spreads through the hypergraph, starting from a set of initially infected vertices, and in each subsequent step of the process every vertex with at least $r$ infected neighbours becomes infected. 
For our analysis the set of initially infected vertices is chosen uniformly at random from all sets of given size.  In the regime $n^{-1}\ll n^{k-2}p \ll n^{-1/r}$ 
we establish a threshold such that if the number of initially infected vertices remains below the threshold, then with high probability only a few additional vertices become infected, while if the number of initially infected vertices exceeds the threshold then with high probability almost every vertex becomes infected. In fact we show that the probability of failure decreases exponentially.
\end{abstract}
\maketitle

\section{Introduction and main result}

The $r$-neighbourhood bootstrap percolation process, which we refer to as $r$-bootstrap percolation from this point on,  
was introduced by Chalupa, Leath, and Reich \cite{bootstrapintr} in the context of magnetic disordered systems. 
Since then $r$-bootstrap percolation and its extensions have been used to describe several complex propogation phenomena: from neuronal activity \cite{MR2728841,inhbootstrap} to the dynamics of the Ising model at zero temperature \cite{Fontes02stretchedexponential}.

In $r$-bootstrap percolation infection spreads over the vertices of a graph, step by step following a deterministic rule. Starting from a set of initially infected vertices, in each subsequent step every uninfected vertex, which has at least $r$ infected neighbours, becomes infected. All other vertices remain uninfected. Once a vertex has become infected, it remains infected forever.  

Kang, Koch and Makai \cite{KANG2015595} introduced an extension of $r$-bootstrap percolation to {\em hypergraphs}. Starting from a set of initially infected vertices in every subsequent step every uninfected vertex becomes infected if it has at least $r$ infected neighbours within the hypergraph, that is, there are at least $r$ infected vertices within the set of hyperedges containing the uninfected vertex. (We note that this definition generalises $r$-bootstrap percolation to the hypergraph setting in a weak sense, and a generalisation in a strong sense will be discussed in Section~\ref{sec:relatedwork}.)

When considering $r$-bootstrap percolation the key question is the relation between the set of initially infected vertices and the set of vertices infected by the end of the process.
Janson,  \L uczak, Turova, and Vallier \cite[Theorem 3.1]{MR3025687} analysed $r$-bootstrap percolation on the binomial random graph $G(n,p)$ where every edge appears independently with probability $p$ and the set of initially infected vertices, denoted by $\infec(0)$, is chosen uniformly at random from all vertex sets of size $a$. They showed, among others, that for $r\geq 2$ and $n^{-1}\ll p \ll n^{-1/r}$ there is a phase transition with threshold $(1-1/r)\left(\tfrac{(r-1)!}{np^r}\right)^{1/(r-1)}$ in terms of the number of initially infected vertices. In~\cite{MR3817532}, this result was strengthened by providing an exponential bound on the error probability. More precisely, we have the following result on the size of $\infec_f$, the set of vertices which become infected during the percolation process.

\begin{theorem}[\cite{MR3025687, MR3817532}]
Given an integer $r\geq 2$, consider $r$-bootstrap percolation  on $G(n,p)$ when $n^{-1}\ll p\ll n^{-1/r}$. Assume that the initial infection set is chosen uniformly at random from the family of all sets of vertices of size $a=a(n)$. Then for any fixed $\eps>0$, with probability at least $1-\exp\left(-\Omega\left((np^r)^{-1/(r-1)}\right)\right)$, the following hold.
\begin{enumerate}[label=$(\roman*)$]
\item  If $a\le (1-\eps)(1-1/r)\left(\tfrac{(r-1)!}{np^r}\right)^{1/(r-1)}$, then $|\infec_f|=O\left(\left(\tfrac{1}{np^r}\right)^{1/(r-1)}\right)$.
\item  If $a\ge (1+\eps)(1-1/r)\left(\tfrac{(r-1)!}{np^r}\right)^{1/(r-1)}$, then $|\infec_f|=(1-o(1))\, n$. 
\end{enumerate}
\end{theorem}

In this paper we generalise this result for the binomial random $k$-uniform hypergraph $H_k(n,p)$, where every hyperedge (which is a $k$-element subset of the vertex set $[n]:=\{1,\ldots, n\}$) is present independently with probability $p$. In this setting, unlike the graph case, the infection of two different vertices is no longer independent, since hyperedges may contain multiple uninfected vertices. The major challenge lies in understanding these dependencies and analysing their impact on the infection process.

\subsection{Main result}

Throughout the paper, for integers $k\geq 2$ and $r\geq 2$ we set
\begin{equation}\label{eq:eta}
 \eta:=\eta(k,r)=
	\begin{cases}
		1 & \mbox{if } r\geq 3, \\
		2k-3 & \mbox{if } r = 2,
	\end{cases}
\end{equation}
and define
\begin{equation}\label{eq:scaling}
	a^*:=a^*(k,r,n,p)=
		\left(\frac{(r-1)!}{\eta n\left(\binom{n}{k-2}p\right)^r}\right)^{1/(r-1)}.
\end{equation}
In addition, we set
\begin{equation}\label{eq:threshold}
a_c:=a_c(k,r,n,p)=\left(1-\frac{1}{r}\right)a^*.
\end{equation}

\begin{theorem}\label{thm:hypergraph}
Given integers $k\geq 2$ and $r\geq 2$,  consider $r$-bootstrap percolation  on $H_{k}(n,p)$ when $n^{-1}\ll n^{k-2} p \ll n^{-1/r}$. Assume the initial infection set is chosen uniformly at random from the family of all sets of vertices of size $a=a(n)$. Then for any fixed $\varepsilon>0$, with probability $1-\exp(-\Omega(a^*))$ the following hold.
	\begin{enumerate}[label=$(\roman*)$]
		\item  If $a\leq(1-\varepsilon)a_c$, then $|A_f|\le a^*$.\label{case:sub}
		\item  If $a\geq (1+\varepsilon)a_c$, then $|A_f|=(1-o(1))\, n$.\label{case:super}
	\end{enumerate}
\end{theorem}

Theorem~\ref{thm:hypergraph} shows in particular  that there is a phase transition with respect to the behaviour of $r$-bootstrap percolation when the number of initially infected vertices reaches $a_c$. If the number of initially infected vertices remains below the threshold $a_c$, then only a few additional vertices become infected. On the other hand, if the number of initially infected vertices is larger than the threshold, then almost every vertex becomes infected by the end of the percolation process. 

\subsection{Related results}\label{sec:relatedwork}
Recall that the key question concerning $r$-bootstrap percolation is to establish the behaviour of the final infection set with respect to the set of initially infected vertices. 
More precisely the typical questions asked are: when does a random subset lead to (almost) every vertex becoming infected, what is the size of the smallest subset which leads to (almost) every vertex becoming infected, and the running time of the process. 
These questions have been investigated over a wide range of graph classes. 
Various deterministic graphs have been considered such as trees \cite{MR2248323,MR2430783}, lattices \cite{MR968311,MR2888224,MR2546747,MR1921442,MR1961342}, hypercubes \cite{MR2214907} and expanders \cite{MR3451155}.
Several random graph models have also been investigated, including the binomial random graph \cite{MR3784495,MR4328426,MR3719944,MR3025687,MR3817532,MR3958417,V07}, the Chung-Lu model \cite{bootpower,arXiv:1402.2815,MR3784494}, preferential attachment graphs \cite{MR3783205} and many others \cite{MR2595485,MR2283230,MR3426518,MR3992287,falgas-ravry_sarkar_2023,MR4206553}. 

Less is known when the underlying structure for $r$-bootstrap percolation is a hypergraph. Recently Cooley and Zalla \cite{CZ2022} analysed the high-order generalisation of $r$-bootstrap percolation in complete $k$-uniform hypergraphs.

A generalisation of $r$-bootstrap percolation to hypergraphs in the strong sense can also be defined. Again starting from a set of infected vertices, in each subsequent step every uninfected vertex $v$ becomes infected if it is contained in at least $r$ hyperedges, where every vertex other than $v$ is infected.
When $r=1$ this model was introduced, under the name of hypergraph bootstrap percolation by Balogh, Bollob\'as, Morris and Riordan \cite{BALOGH20121328}.
In addition to 1-boostrap percolation, hypergraph bootstrap percolation generalises a second bootstrap process called {\em graph bootstrap percolation}, which was introduced by Bolob\'as \cite{MR0244077} under the name weak saturation in 1968. Graph bootstrap percolation acts on the edge set of a graph. More precisely, for a given graph $F$, in each step of the process every edge is added which is the last edge required to create a new copy of $F$. The connection between the above two percolation processes can be observed, when the edges in the graph correspond to vertices in the hypergraph and the hyperedges correspond to the edge sets in copies of $F$.

Graph bootstrap percolation, and consequently hypergraph boostrap percolation, have received significant attention: see e.g.,\ \cite{MR3784495,MR3668849,MR4372097,MR4278608} and the references within.

\subsection{Heuristics for the percolation threshold}\label{ssec:heuristic}
In the remainder of the paper we assume that  $k,r\in \mathbb N_{\geq 2}$  are given and $n\in \mathbb N,  p\in (0,1)$ satisfy $$n^{-1}\ll n^{k-2} p \ll n^{-1/r}.$$
We call a $k$-element subset of $[n]$ simply a {\em $k$-set}.

There is a natural heuristic argument for the evolution of the number of infected vertices over time, which leads to the percolation threshold $a_c$ as in \eqref{eq:threshold}.

Recall that $\infec(0)$ denotes the set of initially infected vertices.  For each $t\in \mathbb N$ we denote by $\infec(t)$ the set of vertices infected by the end of step $t$.
Our heuristic begins with an assumption that the evolution of $\infec(t)$ would be well-approximated by a smooth function $\varphi$, whose behaviour characterises the existence of a \emph{bottleneck} in the process. We provide a heuristic suggesting both what $\varphi$ should be and why this leads to the choice of $a^*$ and the percolation threshold $a_c$. For simplicity we focus on the case $r\ge 3$. In step $t+1\in \mathbb N$, a typical vertex becomes infected due to being contained in exactly $r$ hyperedges where each of these hyperedges contains precisely one (distinct) vertex from $\infec(t)$ and in addition at least one of these hyperedges contains a vertex in $\infec(t)\setminus \infec(t-1)$. Thus, at least heuristically speaking, the approximation
\begin{align*}
\mathbb{E}\left[|\infec(t+1)|\cond \infec(t),\ldots,\infec(0)\right] &\approx |\infec(t)|+n\sum_{i=1}^{r}\frac{(|\infec(t)|-|\infec(t-1)|)^{i}|\infec(t-1)|^{r-i}}{i! (r-i)!}\left(\binom{n}{k-2}p\right)^r
\end{align*}
should be valid as long as only a small fraction of the vertices are infected. In fact, the right-hand side simplifies to $|\infec(t)|+\frac{|\infec(t)|^r-|\infec(t-1)|^r}{r!}n\left(\binom{n}{k-2}p\right)^r$. Iterating yields a telescoping sum and we obtain
\begin{equation}\label{eq:forwardstep}
\mathbb{E}\left[|\infec(t+1)|\cond \infec(t),\ldots,\infec(0)\right] \approx |\infec(0)|+\frac{|\infec(t)|^r}{r!}n\left(\binom{n}{k-2}p\right)^r,
\end{equation}
where we used $\infec(-1):=\emptyset$. Assuming further that $|\infec(t)|$ will be well-approximated by its mean we obtain the forward difference equation 
\begin{equation}\label{eq:forwardDifference}
|\infec(\lfloor (x+1/\timescale)\timescale\rfloor)|-|\infec(\lfloor x\timescale\rfloor)|\approx\frac{|\infec(\lfloor x\timescale\rfloor)|^r}{r!}n\left(\binom{n}{k-2}p\right)^r - |\infec(\lfloor x\timescale\rfloor)| + |\infec(0)|,
\end{equation}
where we have rescaled time as $t=\lfloor x\timescale\rfloor$ for some (arbitrary) $\timescale=\timescale(n)\gg 1$. Now recall that the $r$-bootstrap percolation process stops if and only if $|\infec(\lfloor (x+1/\timescale)\timescale\rfloor)|=|\infec(\lfloor x\timescale\rfloor)|$, in other words, when the forward difference vanishes. 
Rescaling by setting $\alpha(x):= |\infec(\lfloor x\timescale\rfloor)|\left(n\left[\binom{n}{k-2}p\right]^r\right)^{\frac{1}{r-1}}$, from \eqref{eq:forwardDifference} we obtain a simpler form
\begin{equation}\label{eq:alphaDifference}
\ninfecrat(x+1/\timescale)-\ninfecrat(x)\approx\frac{\ninfecrat(x)^r}{r!} - \ninfecrat(x) + \ninfecrat(0).
\end{equation}
Now we observe that if $\ninfecrat$ were to follow a smooth trajectory $\varphi\colon [0,\infty)\to \mathbb{R}$, i.e.,\ $\ninfecrat(x)\approx \varphi(x)$, its derivative $g'$ should obey $\varphi'(x)\approx\ninfecrat(x+1/\timescale)-\ninfecrat(x)$ and thus $\varphi$ should satisfy the following differential equation
\begin{equation}\label{eq:ODE}
\varphi'(x)=\frac{\varphi(x)^r}{r!} - \varphi(x) + \varphi(0),
\end{equation}
where \begin{equation}\label{eq:f(0)}
	\varphi(0)=|\infec(0)|\left(n\left[\binom{n}{k-2}p\right]^r\right)^{\frac{1}{r-1}}. 
\end{equation}	
As mentioned before the process stops, when $\varphi'(x)=0$ and in the following we will examine for which values of $\varphi(0)$ does $\varphi'(x)$ have a non-negative root. 

In order to achieve this we will examine whether the process has a bottleneck, to be characterised by an inflection point, which would translate to an extremal point of $\varphi$ satisfying $\varphi''(x_c)=0$.
By differentiating $\varphi'(x)$ once more, we obtain 
$$
\varphi''(x)=\left(\frac{\varphi(x)^{r-1}}{(r-1)!} - 1\right) \varphi'(x).
$$
Recalling our assumption that $\ninfecrat(x)\approx \varphi(x)$ and the fact that $\ninfecrat(\cdot)$ is  positive and non-decreasing,  we would also assume  $\varphi'(x)\ge 0$. In addition we expect the process to stop if $\varphi'(x)=0$, so we are only interested in the behaviour of the function as long as $\varphi'(x)>0$. 
Therefore 
the bottleneck point should be defined as the earliest point in time such that  $\frac{\varphi(x)^{r-1}}{(r-1)!} - 1=0$, i.e., 
$$
x_c:=\inf\left\{x\ge 0\colon \varphi(x)=\left[(r-1)!\right]^{1/(r-1)}\right\}.
$$
As $r$-bootstrap percolation is monotone in the number of initially infected vertices, intuitively the threshold for almost every vertex becoming infected should be when the process stops right as it reaches the bottleneck. That is,
$$\frac{\varphi(x_c)^r}{r!} - \varphi(x_c) + \varphi(0)=0.$$
As we are assuming the bottleneck exists, that is, $x_c<\infty$, a simple calculation implies that this occurs when
$$\varphi(0)=\varphi_c:=\left(1-\frac{1}{r}\right)\left[(r-1)!\right]^{1/(r-1)}.$$
In addition, the above implies that, if $\varphi(0)\le \varphi_c$ in \eqref{eq:ODE} then $\varphi'(x)$ has at least one non-negative root, while if $\varphi(0)>\varphi_c$ no such root exists.
Using \eqref{eq:f(0)},  ultimately the above indicates a threshold for $r$-bootstrap percolation at
$$a_c=\left(1-\frac{1}{r}\right)\left(\frac{(r-1)!}{n\left(\binom{n}{k-2}p\right)^r}\right)^{1/(r-1)},$$
matching the result from \eqref{eq:threshold}. 

\subsection{Proof outline}

When the number of infected vertices is small, that is, $o((n^{k-2}p)^{-1})$, typically infection occurs when $r$ hyperedges intersect in exactly one uninfected vertex and each of these hyperedges contains exactly one infected vertex. For any $r\ge 2$ this implies that the vertex in the intersection of the hyperedges becomes infected. However, if $r=2$, then in the following step both of these hyperedges will contain 2 infected vertices and thus every other vertex in these hyperedges will be infected as well, leading to $2k-3$ infected vertices in total instead of one and to the definition of $\eta$ in \eqref{eq:eta}.

In the subcritical regime, when $a<a_c$, 
our aim is to show that at most $a^*$ vertices become infected. 
We start in Section~\ref{sec:uppercoupling1} by modifying the process, in a manner that leads to more infected vertices (Lemma~\ref{lem:bootstrapsubset}), but is easier to analyse. 

In Section~\ref{sec:subcritical} we analyse the behaviour of this modified process.
The FKG inequality (Theorem~\ref{thm:FKG}) implies that the number of vertices infected in a step is stochastically dominated by the sum of weighted Bernoulli random variables (Corollary~\ref{lem:stochdom}). 
Next we provide a trajectory, which resembles the solution of the differential equation in \eqref{eq:ODE}, and prove that it also has a maximum (Claim~\ref{lem:boundx}). As long as the expected number of vertices infected in each step is large, which occurs in the early stages of the process, we analyse the process step by step. In every step we show that, provided the number of infected vertices remained below the trajectory in each of the previous steps this will also hold in the following step with sufficiently large probability (Lemma~\ref{lem:upperstep}). 

When the number of vertices infected in a single step is small this strategy no longer provides sufficiently tight probability bounds and we establish an upper bound on the total number of vertices infected after this point. Note that any vertex which becomes infected in step $t+1$ must have a neighbour which became infected in step $t$. 
Based on this we can construct a forest, by selecting exactly one neighbouring vertex which was infected in the previous step. This forest is generated by a branching process, where the number of offspring for the individual vertices depends on the size of the previous generations, that is, the previously infected vertices. Nevertheless, based on the Dwass identity for Galton-Watson branching processes (Theorem~\ref{thm:Dwass}) we show  
that the total number of infected vertices remains small (Lemma~\ref{lem:branching}). 

In the supercritical regime, when $a>a_c$, we differentiate between the vertices depending on the number of infected neighbours it has, more precisely we consider the sets of vertices with at least $0,\ldots,r-1,r$ infected neighbours. Then the set of infected vertices is the union of the set of initially infected vertices and the set of vertices with at least $r$ infected neighbours. Similarly to the subcritical regime we establish a trajectory for the number of vertices in each of these parts, where the trajectory for the number of infected vertices resembles  \eqref{eq:ODE}. As a first step we show that until $(n^{k-2}p)^{-1}$ vertices become infected the number of vertices in each part remain above their respective trajectories with sufficiently large probability. Note that the number of additional infected neighbours the individual vertices acquire in a given step are dependent random variables.
We adapt a Chernoff-type bound for {\em dependent} random variables (Lemma~\ref{lem:negcorr}) to establish the probability that the number of vertices in the individual parts remain above their trajectories after every step (Lemma~\ref{lem:uindstep}). The number of infected vertices will eventually exceed $(n^{k-2}p)^{-1}$ (Lemma~\ref{lem:mininf}), at which point our estimates break down. A slightly modified versions of the concentration result (Lemma~\ref{lem:penultimatestep}) and a martingale argument (Lemma~\ref{lem:finalstep}) imply that once the number of infected vertices reaches $(n^{k-2}p)^{-1}$ almost every vertex becomes infected in the following two steps with sufficiently high probability. These results are shown in Section~\ref{sec:supercritical}.

\section{Preliminaries}

The following concentration result of McDiarmid~\cite{MR1678578} will be used throughout the paper.
\begin{theorem}\label{thm:indconc}\cite[Theorem~2.7]{MR1678578}
	Let $(X_i)_{1\leq i\leq n})$ be a sequence of   independent random variables.  Define  $X:=\sum_{i=1}^n X_i$. If there exists an $M\in \mathbb{R}^+$ such that for each $1\leq i \leq n$, $X_i\leq \mathbb{E}[X_i]+M$, then for all $\vartheta\ge 0$ we have
	$$\mathbb{P}[X\geq \mathbb{E}[X]+\vartheta]\leq \exp\left(-\frac{\vartheta^2}{2\left(\mathrm{Var}(X)+M\vartheta/3 \right)}\right).$$
\end{theorem}

The above inequality holds for independent random variables. We can prove that a similar bound also holds for  {\em dependent} Bernoulli random variables satisfying mild conditions, which strengthens an earlier result due to Panconesi and Srinivasan~\cite{MR1438520}.

\begin{lemma}\label{lem:negcorr}
	For $1\leq i \leq n$ let $X_i$ be (not necessarily independent) Bernoulli random variables. Assume that there exists  a constant  $\lambda>0$ and a set of independent Bernoulli random variables $\widehat{X}_i$ for $1\leq i\leq n$ such that for every $I\subseteq [n]$ we have
	$$\mathbb{E}\left[\prod_{i\in I} X_i\right]\leq \lambda \mathbb{E}\left[\prod_{i\in I} \widehat{X}_i\right].$$
	Define $X:=\sum_{i=1}^n X_i$ and $\widehat{X}:=\sum_{i=1}^n\widehat{X}_i$. Then,  for all $\vartheta\ge 0$ we have
	$$\mathbb{P}[X\geq \mathbb{E}[\widehat{X}]+\vartheta]\leq \lambda \exp\left(-\frac{\vartheta^2}{2(\mathrm{Var}(\widehat{X})+\vartheta/3)}\right).$$
\end{lemma}
The proof of this result can be found in Appendix~\ref{sec:conc}. 
The following result on the total number of progeny in a branching process also plays an important role.
\begin{theorem}[Dwass Identity]\label{thm:Dwass}\cite{MR0253433}
For an arbitrary Galton-Watson process let $Z_i$ denote the number of individuals in the $i$-th generation 
and let $Z:=\sum_{i=0}^{\infty} Z_i$. Then, for all integers $m,\ell\ge 0$ we have
$$\mathbb{P}[Z=m|Z_0=\ell]=\frac{\ell}{m}\mathbb{P}[Z_1=m-\ell|Z_0=m].$$
\end{theorem}

We also apply the FKG inequality (see e.g.,~\cite[Theorem 2.12]{JLR}) several times in our proofs. A hypergraph property is called \emph{increasing} if it is preserved under the addition of hyperedges, and it is \emph{decreasing} if it is preserved under the removal of hyperedges.
\begin{theorem}[FKG inequality]\label{thm:FKG}
	Let $A$ be an increasing hypergraph property and $B$ be a decreasing hypergraph property. Denote by $\mathcal{A}$ and $\mathcal{B}$ the events that, in the random $k$-uniform hypergraph where each hyperedge is selected independently with some probability $p_e$ has property $A$ and $B$, respectively. Then we have
	$$\mathbb{P}[\mathcal{A} \cap \mathcal{B}]\leq \mathbb{P}[\mathcal{A}]\mathbb{P}[\mathcal{B}].$$
\end{theorem}

Finally we will need the well-known Azuma's inequality (see e.g.,~\cite[Theorem 2.25]{JLR})
\begin{theorem}[Azuma's inequality]\label{thm:azuma}
Let $(M_t)_{t\in \mathbb{N}_0}$ be a real-valued supermartingale and let $N\in \mathbb{N}$. Suppose that $|M_i-M_{i-1}|\leq c_i$ for all $1\leq i\leq N$. Then for all $\vartheta \ge 0$ we have
\[
    \mathbb{P}[X_N-X_0\ge \vartheta]\leq \exp\left(-\frac{\vartheta^2}{2\sum_{i=1}^{N}c^2_i}\right).
\]
\end{theorem}

\section{Preliminaries for subcritical regime: the query-process}\label{sec:uppercoupling1}
In Sections~\ref{sec:uppercoupling1} and \ref{sec:subcritical} we consider the subcritical case, when the initial number of infected vertices is below $a_c$. More precisely, throughout these sections  we assume that  a constant $0<\varepsilon<1$ is given and the size of the initial infection set satisfies
$$a:=|\infec(0)|=(1-\varepsilon)a_c.$$

In order to simplify the analysis, we modify the process slightly, in such a way that the modified process provides a {\em superset} of the
set of vertices infected in the original process (Section~\ref{sec:uppercoupling}). 
We will refer to this modified process as the {\em query-process}.
The query-process queries collections of $k$-sets, where we establish whether every $k$-set in a collection is a hyperedge or not. Vertices become infected only if every hyperedge within a collection is present.

We show that starting from the same infection set every vertex which is infected via $r$-bootstrap percolation is also infected by the query-process (Lemma~\ref{lem:bootstrapsubset}). We also establish that the trajectory for the number of infected vertices in the query-process remains below $a^*$ (Lemma~\ref{lem:ubound}).

The main advantage of the modification is that the probability whether every hyperedge in a collection is present in the hypergraph is stochastically dominated by a sequence of independent random variables, which depend only on the hyperedge probability $p$ and the number of hyperedges in the collection (Lemma~\ref{lem:indepbound}) and thus the number of infected vertices is stochastically dominated by the sum of weighted Bernoulli random variables (Lemma~\ref{lem:stochdom}). 
However, the number of {\em queries} depends on the history of the process, more precisely it depends on the number of vertices infected in the previous step and the total number of vertices infected until this step. Based on this we can show that, as long as the number of infected vertices remains below a certain trajectory in each of the previous steps and the expected number of infected vertices is large enough, the number of infected vertices in the next step will also remain below the trajectory (Lemma~\ref{lem:upperstep}). 

However, once the expected number of infected vertices becomes small, the previous method no longer provides a sufficiently  tight concentration. Note that any vertex which becomes infected in step $t$ must have a neighbour which became infected in step $t-1$. 
In fact, once the number of vertices infected in a given step is sufficiently small compared to the total number of infected vertices, a typical vertex will have exactly one such neighbour. Based on this we can construct a forest, by connecting every infected vertex to exactly one of its neighbours which was infected during the previous step. This forest is generated by a process resembling a branching process, with the main difference that the number of offspring depends on the size of the previous generations. Using the Dwass identity (Theorem~\ref{thm:Dwass}) we show that the number of vertices which become infected during the remaining stages of the process is small (Lemmas~\ref{lem:branching}). 

\subsection{An upper coupling}\label{sec:uppercoupling}
We start by describing an upper coupling, which we call the {\em query-process}, informally and thereby introducing terminologies that will be used later.

The query-process is also an infection process which starts out from a set of initially infected vertices. 
At the beginning of each step we construct multiple families, where each family is a collection of $k$-sets. 
The rest of the step is divided into rounds, where in each round we \emph{examine} one collection in the family. First it is decided whether the collection of $k$-sets will be \emph{queried} based on the results of all the previous queries made during the current and the previous steps. If a collection is not queried then it is \emph{discarded}. A query is \emph{successful} if every $k$-set in the collection is a hyperedge. Every hyperedge within a successful query is called \emph{exposed}.

Now we turn to the construction of the family of collections of $k$-sets. Essentially our aim is to infect every vertex which has at least $r$ infected neighbours or this should happen if every collection in the family was queried. In addition, certain collections are added to the family in order to compensate for vertices with $r$ infected neighbours, which can be found in discarded collections. The constructions are displayed in Figure~\ref{fig:constructions}.

We start by adding every collection of $r$ many  $k$-sets which intersect in exactly one vertex $v$ and each $k$-set contains exactly one previously infected vertex. We call such a collection a {\em star} collection (e.g., the triple of the dotted $k$-sets in Figure~\ref{fig:constructions}(a). Any successful query of a star collection will result either in every vertex within the $k$-set becoming infected (when $r=2$) or in the vertex $v$ becoming infected (when $r\ge 3$). In a subcritical query-process 
almost every vertex infected in a given step is contained in a star.

Now if a vertex has $r$ infected neighbours, which are not in a star collection, then either two of the hyperedges overlap in more than one vertex or a hyperedge contains multiple infected vertices.

We add every collection of pairs of $k$-sets which overlap in at least 2 uninfected vertices and both hyperedges contain at least one infected vertex. Such a collection is called a {\em widely-overlapping} collection (e.g., the pair of the dotted $k$-sets in Figure~\ref{fig:constructions}(b)). Successfully querying such a collection leads to the infection of every vertex in both of the $k$-sets.

Afterwards we add every collection of $k$-sets containing at least two infected vertices, which have not been exposed until this step. When successfully queried, every vertex in the $k$-set becomes infected. Such a collection is called a {\em heavily-infected} collection (e.g., the dotted $k$-set in Figure 1(c)).

Finally we consider the hyperedges which have been exposed and still contain uninfected vertices. Note that only hyperedges in stars when $r\ge 3$ can contain uninfected vertices. We add every collection consisting of a $k$-set, containing an infected vertex which intersects an exposed hyperedge in an uninfected vertex.
Such a collection is a so-called {\em neutron-star-adjacent} (e.g., the dotted $k$-set in Figure 1(d)).
If such a hyperedge is successfully queried, then we infect every vertex in the hyperedge and also every vertex in any exposed hyperedge it intersects in an uninfected vertex.

We will introduce a final criteria to reduce the number of times a collection of $k$-sets is added to each family. Ideally every collection of $k$-sets would be added to each family at most once, e.g.,\ once a star is added it will never be added again.  Note that this still allows for collections to be added to a different family, e.g.,\ if a collection consisting of a single $k$-set was added to the family of heavily-infected collections it can be later added to the family of neutron-star-adjacent collections. Instead of working with the complicated conditions required for each collection to be added at most once, we use simpler, but weaker conditions, in the sense that the number of times a collection is added is significantly reduced, but some collection of $k$-sets can be added multiple times.  

If none of the collections has been discarded, or equivalently if every collection was queried, then every vertex which has at least $r$ infected neighbours would become infected. However, in order to simplify the analysis of the process, some of these collections will not be queried. Thus the query process might contain uninfected vertices with $r$ infected neighbours. In Lemma~\ref{lem:bootstrapsubset} we show that any such vertex will become infected at a later stage in the process. 

\begin{figure}[ht]
	\begin{center}
		\begin{tikzpicture}[scale=0.7]
			
			\node (S0) at (-4,0) [circle, draw, inner sep=2pt] {};
			
			\node (S1) at (-5,-1) [circle, draw, inner sep=2pt] {};
			\node (S2) at (-6,-2) [circle, draw, inner sep=2pt] {};
			\node (S3) at (-7,-3) [circle, fill=black, inner sep=2pt] {};
			
			\node (S4) at (-5,0) [circle, draw, inner sep=2pt] {};
			\node (S5) at (-6,0) [circle, draw, inner sep=2pt] {};
			\node (S6) at (-7,0) [circle, fill=black, inner sep=2pt] {};
			
			\node (S7) at (-5,1) [circle, draw, inner sep=2pt] {};
			\node (S8) at (-6,2) [circle, draw, inner sep=2pt] {};
			\node (S9) at (-7,3) [circle, fill=black, inner sep=2pt] {};
			
			\draw[black,dashed,very thick]\convexpath{S0,S3}{0.27cm};
			\draw[black,dashed,very thick]\convexpath{S0,S6}{0.27cm};
			\draw[black,dashed,very thick]\convexpath{S0,S9}{0.27cm};
			
			\node at (-6,-4.5){(a) a star collection};

			\node (O0) at (1,0) [circle,draw, inner sep=2pt] {};
			\node (O1) at (0,0) [circle,draw, inner sep=2pt] {};
			
			\node (O2) at (-1,-1) [circle,draw, inner sep=2pt] {};
			\node (O3) at (-2,-2) [circle,fill=black, inner sep=2pt] {};
			
			\node (O4) at (-1,1) [circle,draw, inner sep=2pt] {};
			\node (O5) at (-2,2) [circle,fill=black, inner sep=2pt] {};
			
			\draw[black,dashed,very thick]\convexpath{O3,O1,O0}{0.27cm};
			\draw[black,dashed,very thick]\convexpath{O5,O0,O1}{0.27cm};
			
			\node at (-0.9,-4.5)[text width=3.5cm]{(b) a widely-overlapping collection};
			
			\node (M0) at (6,0) [circle,draw, inner sep=2pt] {};
			\node (M1) at (5,0) [circle, draw, inner sep=2pt] {};
			\node (M2) at (4,0) [circle, fill=black, inner sep=2pt] {};
			\node (M3) at (3,0) [circle, fill=black, inner sep=2pt] {};
			
			\draw[black,dashed,very thick]\convexpath{M0,M3}{0.27cm};
			
			\node at (4.5,-4.5)[text width=3cm]{(c) a heavily-infected collection};
			
			\node (L0) at (11,0) [circle,fill=black, inner sep=2pt] {};
			
			\node (L1) at (10,-1) [circle, draw, inner sep=2pt] {};
			\node (L2) at (9,-2) [circle, draw, inner sep=2pt] {};
			\node (L3) at (8,-3) [circle, fill=black, inner sep=2pt] {};
			
			\node (L4) at (10,0) [circle, draw, inner sep=2pt] {};
			\node (L5) at (9,0) [circle, draw, inner sep=2pt] {};
			\node (L6) at (8,0) [circle, fill=black, inner sep=2pt] {};
			
			\node (L7) at (10,1) [circle, draw, inner sep=2pt] {};
			\node (L8) at (9,2) [circle, draw, inner sep=2pt] {};
			\node (L9) at (8,3) [circle, fill=black, inner sep=2pt] {};
			
			\node (L10) at (8,2) [circle, fill=black, inner sep=2pt] {};
			\node (L11) at (9,2) [circle, draw, inner sep=2pt] {};
			\node (L13) at (10,2) [circle, draw, inner sep=2pt] {};			
			\node (L12) at (11,2) [circle, draw, inner sep=2pt] {};
			
			\draw[black,very thick]\convexpath{L0,L3}{0.27cm};
			\draw[black,very thick]\convexpath{L0,L6}{0.27cm};
			\draw[black,very thick]\convexpath{L0,L9}{0.27cm};
			\draw[black,dashed,very thick]\convexpath{L10,L12}{0.27cm};
			
			\node at (9.5,-4.5)[text width=3cm] {(d) a neutron-star adjacent collection};

		\end{tikzpicture}
		\caption{Collections of the query process.\\ Filled circles represent infected vertices, empty circles represent uninfected vertices.\\
			Ovals represent exposed hyperedges, dashed ovals represent a collection of $k$-sets.}
		\label{fig:constructions}
	\end{center}
\end{figure}

\subsection{The query-process}\label{sec:query-process}
Consider a $k$-uniform hypergraph $H(V,E)$ with vertex set $V$ and hyperedge set $E$. We proceed to define the query-process formally. 
To this end we need further definitions and notation that will be used to describe and analyse the query-process.

Let $\linfec(0)$ denote the initial infection set in the  query-process.  For each $t\in \mathbb N$, denote by  
\begin{itemize}
	\item  $\linfec(t)$ the set of vertices infected by step $t$ (and its complement $\linfec(t)^{C}$);
	\item  $\infedge(t)$  the set of hyperedges exposed before the end of step $t$.
\end{itemize}

When considering the $(t+1)$-st step we rely on the sets $\linfec(t),\linfec(t-1),\infedge(t)$ and $\infedge(t-1)$. Note that $\infedge(0)=\emptyset$, and we extend the domain by setting $\linfec(-1)=\emptyset$,  $\infedge(-1)=\emptyset$ in order to cover the $t=0$ case.

Before providing the algorithm, we introduce the rigorous definitions for each family of collections of $k$-sets associated with the query-process we saw in Section~\ref{sec:uppercoupling}: star collections, widely-overlapping collections, heavily-infected collections and neutron-star-adjacent collections in Definitions~\ref{def:stars}, \ref{def:overlapping}, \ref{def:overinfected} and \ref{def:intersecting}.

First we deal with the family of star collections. Recall that a star consists of $r$ many $k$-sets which intersect in exactly one vertex and each $k$-set contains exactly one infected vertex. This is covered by the first three conditions of Definition~\ref{def:stars}, the first two conditions focus on the structural properties, while the third condition establishes that the required infected vertices are present. The final condition is there to ensure that we only consider collections, which have not been added as a star collection during a previous step, which we achieve by insisting that one of the infected vertices, within the $k$-sets of the star collection, was infected in the previous step.
\begin{definition}[Star collections; see Figure~\ref{fig:constructions}(a)]\label{def:stars}
	For each $t\in \mathbb N$, let 
	$\dom(t+1)$ be the family of $r$ many $k$-sets $\coll=\{\kset_1,\ldots,\kset_r\}$ satisfying the following conditions: 
	\begin{itemize}
		\item  there exists a unique uninfected vertex $v\in \linfec(t)^C$ within the intersection of the $k$-sets, that is, $\{v\}=\cap_{\kset\in\coll} \kset$; 
		\item any two $k$-sets  in  $\coll$ intersect in exactly one vertex;
		\item  every $\kset\in \coll$ contains exactly one infected vertex, that is, $|\kset\cap\linfec(t)|=1$;
		\item  there exists a $\kset\in \coll$, containing a vertex infected in the previous step, that is, $\kset\cap \linfec(t)\backslash\linfec(t-1)\neq \emptyset$.
	\end{itemize}
	We call  $\dom(t+1)$  {\em the family of star collections (at step $t+1$ of the query-process)}.
\end{definition}

We continue with the family of widely-overlapping collections. Similarly to star collections the first condition describes the structural property, that is, we have two hyperedges which overlap in at least 2 uninfected vertices, the second condition establishes that every hyperedge contains at least one infected vertex, while the last condition is there to reduce the number of times a collection is considered (as a widely-overlapping collection), that is, one of the $k$-sets must contain a vertex infected in the previous step.
\begin{definition}[Widely-overlapping collections; see Figure~\ref{fig:constructions}(b)]\label{def:overlapping}
	For each $t\in \mathbb N$, denote 
	by $\overlap(t+1)$ the family of pairs $\{\kset_1,\kset_2\}$ of $k$-sets  satisfying the following conditions:
	\begin{itemize}
		\item $\kset_1$ and $\kset_2$ overlap in at least 2 uninfected vertices, that is, $|\kset_1\cap\kset_2\cap\linfec(t)^C|\ge 2$;
		\item every $K\in \{K_1,K_2\}\}$ contains at least one infected vertex that is, $|K\cap B(t)|\geq 1$;
		\item $\kset_1$ or $\kset_2$ contains a vertex infected in the previous step, that is $(K_1\cup K_2)\cap(\linfec(t)\backslash\linfec(t-1))\neq \emptyset$. 
	\end{itemize}
	We call  $\overlap(t+1)$  {\em the family of widely-overlapping collections (at step $t+1$ of the query-process)}.
\end{definition}

Next we consider the family of heavily-infected collections. This time there are no structural conditions as the collections consist of exactly one $k$-set. The first condition establishes that the $k$-set contains multiple infected vertices, while the last condition is there to reduce the number of times a collection is added to the family of heavily-infected collections, that is, at least one of the vertices was infected in the previous step.
\begin{definition}[Heavily-infected collections; see Figure~\ref{fig:constructions}(c)]\label{def:overinfected}
	For each $t\in \mathbb N$, let 
	$\multi(t+1)$ be the family of $k$-sets $\kset$ which satisfy the following conditions:
	\begin{itemize}
		\item $\kset$ contains at least two infected vertices, that is, $|\kset\cap \linfec(t)|\ge 2$;
		\item $\kset$ contains at least one vertex infected in the previous step, that is, $\kset \cap \linfec(t)\backslash \linfec(t-1)\neq \emptyset$.
	\end{itemize}
	We call  $\multi(t+1)$   {\em the family of  heavily-infected collections (at step $t+1$ of the query-process)}.
\end{definition}

Finally we consider the family of neutron-star-adjacent collections. In contrast to the previous families, which only depend on the set of infected vertices at the end of the last two steps, this also depends on the set of exposed hyperedges at the end of the last two steps.
Here we are considering $k$-sets, which intersect an exposed hyperedge in an uninfected vertex. Unlike the previous cases, in order to reduce the number of times such a $k$-set is considered, we not only have to consider $k$-sets which contain vertices infected during the previous step, but also $k$-sets which intersect hyperedges exposed during the previous step.  
\begin{definition}[Neutron-star-adjacent collections; see Figure~\ref{fig:constructions}(d)]\label{def:intersecting}
	For each $t\in \mathbb N$, let  
	$\cross_1(t+1)$ denote the family of $k$-sets $\kset$ which satisfy the following conditions:
	\begin{itemize}
		\item $\kset$ contains a vertex infected in the previous step, that is, $\kset \cap \linfec(t)\backslash \linfec(t-1)\neq \emptyset$;
		\item there exists an uninfected vertex $u\in \linfec(t)^C$ such that $u\in K$ and $u$ is incident to an exposed hyperedge $e\in \infedge(t)$.
	\end{itemize}
	In addition, let $\cross_2(t+1)$ denote the family of $k$-sets $K$, which satisfy the following conditions:
	\begin{itemize}
		\item $\kset$ contains at least one infected vertex, that is, $\kset\cap \linfec(t)\neq \emptyset$;
		\item there exists an uninfected vertex $u\in \linfec^C(t)$ such that $u\in \kset$ and $u$ is incident to a hyperedge $e$ exposed in the previous step, that is, $e\in\infedge(t)\backslash\infedge(t-1)$.
	\end{itemize}
	Set $\cross(t+1):=\cross_1(t+1)\cup\cross_2(t+1)$. 
	We call  $\cross(t+1)$   {\em the family of  neutron-star-adjacent collections (at step $t+1$ of the query-process)}.
	
\end{definition}

Having established the families, Algorithm~\ref{alg:query-process} describes one step of the query-process.

\begin{algorithm} 
	\KwIn{$H(V,E),\,\dom(t+1),\overlap(t+1),\multi(t+1),\cross(t+1),\linfec(t),\infedge(t)$}
	$\Phi \leftarrow \infedge(t)$\\
	$B\leftarrow \linfec(t)$\\
	\ForAll{$\kset\in \cross(t+1)$}
	{ \If {$\kset$ is not a previously exposed hyperedge ($\kset\not\in \Phi$) \label{ln:crossexp}} 
		{ \If{$\kset\in E$ \label{ln:crossedge}}
			{The hyperedge $\kset$ is exposed ($\Phi\leftarrow \Phi \cup \{\kset\}$)\\
				Infect every vertex in every hyperedge $e\in \infedge(t)$ incident to an uninfected vertex in $K$ \label{ln:infcross1}\\
				Infect every vertex in $K$ \label{ln:infcross2}\\
				($B\leftarrow B \cup \kset \cup \{u\in e: e\in \infedge(t)\quad and \quad \kset \cap e \cap \linfec(t)^C\neq \emptyset\}$)
			}
		}
	}
	\ForAll{$\kset\in \multi(t+1)$}
	{ \If {$\kset$ is not a previously exposed hyperedge ($\kset\not\in \Phi$) \label{ln:multiexp}}
		{ \If{$\kset\in E$ \label{ln:multiedge}}
			{The hyperedge $\kset$ is exposed ($\Phi\leftarrow \Phi \cup \kset$)\\
				Infect every vertex in $\kset$ ($B\leftarrow B \cup \kset$)\label{ln:infmulti}}
		}
	}	
	\ForAll{$\{\kset_1,\kset_2\}\in  \overlap(t+1)$}
	{ \If {neither $\kset_1$ nor $\kset_2$ is an exposed hyperedge ($\{\kset_1,\kset_2\}\cap \Phi = \emptyset$)\label{ln:overlapexp}}
		{ \If{$\{\kset_1,\kset_2\}\subseteq E$\label{ln:overlapedge}}
			{The hyperedges $\kset_1$ and $\kset_2$ are exposed ($\Phi\leftarrow \Phi\cup \{\kset_1,\kset_2\}$)\\
				Infect every vertex in $\kset_1$ and $\kset_2$ ($B\leftarrow B \cup \kset_1\cup \kset_2$)\label{ln:infoverlap}}
		}	
	}
	\ForAll{$\{\kset_1,\ldots,\kset_r\}\in \dom(t+1)$}
	{ \If {none of the $k$-sets in $\{\kset_1,\ldots,\kset_r\}$ has been previously exposed ($\{\kset_1,\ldots,\kset_r\}\cap \Phi=\emptyset$) \label{ln:domexp}}
		{\If{for every $e\in \Phi$ and $\kset \in \{\kset_1,\ldots,\kset_r\}$ we have $e\cap K \cap \linfec(t)^C= \emptyset$ \label{ln:domexpint}}	
			{ \If{$\{\kset_1,\ldots,\kset_r\}\subseteq E$ \label{ln:domedge}}
				{The hyperedges in $\{\kset_1\ldots,\kset_r\}$ are exposed ($\Phi \leftarrow \Phi\cup \{\kset_1,\ldots,\kset_r\}$ \label{ln:expdom})\\
					\uIf {$r=2$}
					{Infect every vertex in $\kset_1$ and $\kset_2$ ($B\leftarrow B \cup \kset_1 \cup \kset_2$)\label{ln:infdom2}}
					\Else {	Infect the vertex in the intersection of the $k$-sets ($B\leftarrow B \cup \bigcap_{i=1}^r \kset_i$)}\label{ln:infdom}}
		}}
	}
	$\linfec(t+1)\leftarrow B$\\
	$\infedge(t+1)\leftarrow \Phi$\\
	\KwOut{$\linfec(t+1),\infedge(t+1)$}
	\caption{Query-process}
	\label{alg:query-process}
\end{algorithm}
 
As previously mentioned not every collection within the individual families is queried. In particular if any of the conditions in lines \ref{ln:crossexp}, \ref{ln:multiexp}, \ref{ln:overlapexp}, \ref{ln:domexp} or \ref{ln:domexpint} fail, for an element of the given family, then the collection is discarded.

\subsection{Auxiliary results}\label{sec:aux-results}
In this section we discuss auxiliary results required to analyse the query-process.

First we establish upper bounds on the sizes of the family $\dom(t+1)$ of the star collections, the family $\overlap(t+1)$ of the widely-overlapping collections, the family $\multi(t+1)$of the heavily-infected collections, and the family $\cross(t+1)$ of the neutron-star adjacent collections. Set 
$$N(t+1):=|\cross(t+1)|, \quad H(t+1):=|\multi(t+1)|, \quad W(t+1):=|\overlap(t+1)|, \quad S(t+1):=|\dom(t+1)|.$$ 
A quick calculation establishes an upper bound on $S(t+1)$.
\begin{obs}\label{obs:starsize}
	We have
	$$S(t+1)
	\le \frac{|\linfec(t)|^r-|\linfec(t-1)|^r}{r!}n \binom{n}{r-2}^r.$$
\end{obs}
Similarly we obtain an upper bound on $W(t+1)$.
\begin{obs}\label{obs:overlapsize}
	We have
	$$W(t+1)
	\le |\linfec(t)|(|\linfec(t)|-|\linfec(t-1)|)n^2 \binom{n}{k-3}^2.$$
\end{obs}
Next we deduce an upper bound on $H(t+1)$.
\begin{obs}\label{obs:multisize}
	We have
	$$H(t+1)
	\le |\linfec(t)|(|\linfec(t)|-|\linfec(t-1)|)\binom{n}{k-2}.$$
\end{obs}

In order to give an upper bound on $N(t+1)$, we include here the following estimate on the number of exposed hyperedges, which could be associated with a neutron-star-adjacent collection.
\begin{claim}\label{clm:infedgetoinfver}
For every $t\ge 0$ the number of exposed hyperedges in $\infedge(t)\setminus \infedge(t-1)$ which still contain an uninfected vertex is at most $r |\linfec(t)\setminus \linfec(t-1)|$.
\end{claim}

\begin{proof}
	By lines \ref{ln:infcross1}, \ref{ln:infcross2}, \ref{ln:infmulti}, \ref{ln:infoverlap}, \ref{ln:infdom2} and \ref{ln:infdom} in Algorithm~\ref{alg:query-process} any exposed hyperedge which still contains uninfected vertices must come from a star when $r\ge 3$. In addition exposing a star, when $r\ge 3$ results in exactly one vertex becoming infected. By Definition~\ref{def:stars} and line~\ref{ln:expdom} of Algorithm~\ref{alg:query-process} each infected vertex is responsible for at most $r$ such exposed hyperedges and the result follows.
\end{proof}
Finally we obtain an upper bound on $N(t+1)$.
\begin{obs}\label{obs:crosssize}
We have
$$N(t+1)
\le 2 k r(|\linfec(t)|-|\linfec(t-1)|)|\linfec(t)|\binom{n}{k-2}.$$
\end{obs}

\begin{proof}
Clearly 
$$N(t+1)\le k |\linfec(t)|(|\infedge(t)|-|\infedge(t-1)|)\binom{n}{k-2} + k|\infedge(t)| (|\linfec(t)|-|\linfec(t-1)|)\binom{n}{k-2}.$$
The result then follows from Claim~\ref{clm:infedgetoinfver}.
\end{proof}

\subsection{Coupling $r$-boostrap percolation and the query-process}\label{sec:coupling}
In this section we show that every vertex infected by $r$-boostrap percolation is also infected by the query-process, assuming they run on the same $k$-uniform hypergraph with the same initial infection set. Note that this is not necessarily true for every step of the process, as the query-process may contain uninfected vertices, which would have already been infected by $r$-bootstrap percolation. 

Roughly speaking the query-process focuses on adding the vertices infected through star collections one at a time. Recall that a star collection consists of $r$ hyperedges each containing exactly one infected vertex such that the hyperedges overlap in exactly one vertex, which has to be uninfected. More precisely these are considered as long as none of the hyperedges in the star collection intersect or match any previously exposed hyperedges. 
Essentially this means that we are excluding infections which result from hyperedges containing multiple infected vertices, hyperedges overlapping in multiple vertices or hyperedges which intersect an already exposed hyperedge. The families of widely-overlapping, heavily-infected,  and neutron-star-adjacent collections are created in order to include the vertices which become infected in such a manner as well. 
In the following lemma we provide a rigorous version of this argument. Recall that $\infec_f$ denotes the set of vertices which become infected during $r$-boostrap percolation. Similarly, let $\linfec_f$ denote the set of vertices which become infected during the query-process.

\begin{lemma}\label{lem:bootstrapsubset}
	Given any (deterministic or random) hypergraph and a set of initially infected vertices,   every vertex infected by $r$-bootstrap percolation will also be infected by the query-process, i.e.,
	$$\infec_f \subseteq \linfec_f.$$
	
\end{lemma}

\begin{proof}
	Assume for contradiction that $\infec_f\backslash \linfec_f\neq\emptyset$. Select $u\in \infec_f\backslash \linfec_f$ such that $u$ becomes infected in step $s+1$, that is, $u\in\infec(s+1)\backslash \infec(s)$ and 
	$\infec(s)\subseteq \linfec_f$. Note that since $\infec(0)=\linfec(0)\subseteq \linfec_f$ such a vertex must exist. This means that even though $u$ has at least $r$ infected neighbours in $\linfec_f$, it remained uninfected. Let $e_1,\ldots,e_{\ell}$ be the set of hyperedges which contain $u$ and at least one vertex in $\infec(s)$. 
	
	\begin{claim}\label{clm:nocrossmulti}
	For any $1\le i \le \ell$ and $t\ge 1$ we have $e_i\not \in \cross(t)\cup \multi(t)$.
	\end{claim}

	\begin{proof}
	Without loss of generality assume for contradiction that $e_1\in \cross(t) \cup \multi(t)$ for some $t\ge 1$.
	 
	If $e_1$ has been successfully queried as a neutron-star-adjacent collection or as a heavily-infected collection, then according to the lines~\ref{ln:crossexp} and \ref{ln:crossedge} of Algorithm~\ref{alg:query-process} when $e_1\in \cross(t)$ and lines~\ref{ln:multiexp} and \ref{ln:multiedge} when $e_1\in\multi(t)$ the following conditions must hold: $e_1$ has not been exposed by the time of the query and $e_1$ must be a hyperedge in the hypergraph. In both of these cases, by lines~\ref{ln:infcross2} and \ref{ln:infmulti} respectively, every vertex in $e_1$ would become infected, a contradiction as $u\in e_1$ remains uninfected until the end of the process. 
	
	By assumption, $e_1\in E$, and thus we must have that the hyperedge $e_1$ was exposed before the query could occur. More precisely, it had to be exposed during an earlier step, that is, $e_1\in \infedge(t-1)$, as we start step $t$ by querying the neutron-star-adjacent collections, which is followed by the heavily-infected collections and in both cases exposing the hyperedge $e_1$ would result in every vertex within $e_1$, including $u$, becoming infected, a contradiction. 	
	
	When $r=2$ then lines~\ref{ln:infcross2},\ref{ln:infmulti},\ref{ln:infoverlap},\ref{ln:infdom2} imply that every vertex within each exposed hyperedge is infected, which implies that $u$ is also infected. This leads to a contradiction and completes the proof of the $r=2$ case.
	
	On the other hand, if $r\ge 3$, then the rule for infecting vertices within star collections changes from infecting every vertex in the exposed hyperedges (line~\ref{ln:infdom2}) to infecting only one vertex (line~\ref{ln:infdom})
	and thus $e_1$ must be in a successfully queried collection of $r$ many $k$-sets $\coll\in \dom(\tau)$, for some $\tau < t$.

	Since $e_1$ is a $k$-set in $\coll$ by Definition~\ref{def:stars} it contains exactly one infected vertex at the beginning of step $\tau$. Now $\coll\in \dom(\tau) $ is successfully queried only if the following conditions hold, no $k$-set $\kset\in \coll$ has been exposed (line~\ref{ln:domexp}) or intersects an exposed hyperedge in an uninfected vertex (line~\ref{ln:domexpint}). 
	As only vertices within exposed hyperedges can become infected, no additional vertices within the $k$-sets in $\coll$ are infected, until $\coll$ is queried. By line~\ref{ln:infdom}, only one vertex is infected, when $\coll$ is successfully queried, namely the vertex which is contained in every $k$-set within $\coll$. At this point $e_1$ will contain two infected vertices, neither of which is $u$. In fact, this will hold until the end of step $\tau$, as after this round only stars which share no uninfected vertices with $e_1$ are queried (lines~\ref{ln:domexp},\ref{ln:domexpint}).

	In the remainder of the proof we will show that $e_1$ is the only hyperedge to contain a vertex in $\linfec_f$ and an uninfected vertex in $e_1$. This in turn implies, that the vertex $u$ has only two infected neighbours, leading to a contradiction as by assumption it has $r\ge 3$ neighbours in $\linfec_f$.
	
	Assume for contradiction that there exists a hyperedge $\hat e$ containing at least one vertex in $\linfec_f$ and at least one uninfected vertex of $e_1$. Note that $\hat e$ is a neutron-star-adjacent collection and would be included either in $\cross_2(\tau+1)$ or in $\cross_1(\tau_1)$ for some $\tau_1>\tau$. Denote by $\tau_1$ the first integer larger than $\tau$ such that  $\hat e\in \cross(\tau_1)$. Similarly as before, if $\hat e$ were queried (line~\ref{ln:crossexp}) as a neutron-star-adjacent collection, then this query would be successful as $\hat e\subseteq E$ (line~\ref{ln:crossedge}) and every vertex of $e_1$, including $u$, would be infected (line \ref{ln:infcross1}). Since the family of neutron-star-adjacent collections are the first to be considered in a given step, there must exist a step $\tau'<\tau_1$ such that $\hat e$ was exposed in that step. 
	
	 Recall that by the end of step $\tau$ no previously exposed hyperedge intersects $e_1$ in an uninfected vertex, and thus $\hat e$ must have been exposed between steps $\tau$ and $\tau_1$, that is, $\tau<\tau'<\tau_1$. Note that if $\hat e$ was exposed in step $\tau'$ it must contain a vertex which became infected in an earlier step, $\tau''<\tau'$. 
	 Now if $\tau''\le \tau$ then $\hat e\in \cross(\tau+1)$ leading to a contradiction as $\tau+1<\tau_1$. On the other hand if $\tau''>\tau$ then $\hat e\in\cross(\tau''+1)$ which is also a contradiction as $\tau''+1\le\tau'<\tau_1$.	
 \end{proof}

	\begin{claim}\label{clm:nooverlap}
		For any $1\le i < j  \le \ell$ and $t\ge 1$ we have $\{e_i,e_j\}\not \in \overlap(t)$.
	\end{claim}
	
	\begin{proof}
	Without loss of generality assume for contradiction that $\{e_1,e_2\}\in \overlap(t)$ for some $t\ge 1$.
	Similarly to the family of neutron-star-adjacent collections and heavily-infected collections, if neither of $\{e_1,e_2\}$ was exposed before it was queried (line~\ref{ln:overlapexp}) and $\{e_1,e_2\}\subseteq E$, then by line~\ref{ln:infoverlap} every vertex in $e_1,e_2$ will be infected. Clearly $e_1,e_2\in E$ and thus at least one of $e_1$ or $e_2$ must have been exposed earlier. 
	In fact this can only occur if $r\ge 3$ and there exists $\coll\in \dom(\tau)$ for some $\tau\leq t$ such that $e_1\in \coll$ or $e_2\in \coll$, as for every other family, every vertex within every exposed hyperedge becomes infected (lines~\ref{ln:infcross2},\ref{ln:infmulti},\ref{ln:infoverlap},\ref{ln:infdom2},\ref{ln:infdom}). Without loss of generality assume that $e_1\in \coll$.
	
	By Definitions~\ref{def:stars} and \ref{def:overlapping} we have $e_2 \not \in \coll$. Recall that $u$, a vertex which remains uninfected throughout the process is found in both of the hyperedges $e_1,e_2$. Since $\coll$ was successfully queried, by line~\ref{ln:domexpint}, the hyperedge $e_2$ has not been exposed before $\coll$ was queried, or during the remainder of step $\tau$. By Definition~\ref{def:overlapping} the hyperedge $e_2$ must contain a vertex other than $u$ which became infected in some step $\tau'<t$. If $\tau'\le \tau$, then $e_2\in \cross(\tau+1)$ otherwise $e_2\in \cross(\tau'+1)$, by the definitions of $\cross_1(\tau+1)$ and $\cross_2(\tau'+1)$ in Definition~\ref{def:overlapping}. Now the only way $e_2$ is queried in a manner that leaves $u$, a vertex of $e_2$, uninfected is, if $e_2$ has been exposed as part of a previously successfully queried star collection. Recall that $e_2$ could not have been exposed before the end of step $\tau$ and by line~\ref{ln:domexpint} it can not be part of an exposed star collection during the later steps of the process. Therefore $e_2$ will be successfully queried, in such a way that all of its vertices become infected, leading to a contradiction. 
	\end{proof}

	By Claim~\ref{clm:nocrossmulti} none of the hyperedges $e_1,\ldots,e_{\ell}$ can be contained in $\multi(t)$ for any $t \ge 1$, implying that each of these hyperedges contains exactly one infected vertex in $\linfec_f$. In addition by Claim~\ref{clm:nooverlap} no pair of hyperedges in $\{e_1,\ldots,e_\ell\}$ was contained in $\overlap(t)$ for any $t\ge 1$ and thus any pair of these hyperedges intersect only in $u$ in $\linfec_f^C$.
	
	Therefore, as $u$ has $r$ infected neighbours in $\infec_f$ there must exist a set $\coll\subseteq \{e_1,\ldots,e_\ell\}$, with $|\coll|=r$ such that any pair of hyperedges in $\coll$ intersect only in $u$. Now $\coll$ will be contained in the family $\dom(t)$ for some $t\ge 1$. Let $\tau$ be the first time this holds and without loss of generality assume that $\coll=\{e_1,\ldots,e_r\}$.
	A successful query of $\coll$ would have infected $u$, implying that either one of the $e_i\in\coll$ has been exposed previously (line~\ref{ln:domexp}) or one of the $e_i\in\coll$ intersects an exposed hyperedge in an uninfected vertex (line~\ref{ln:domexpint}). 
	
	Now if $e_i$ is exposed and $u$ is uninfected, then every other hyperedge $e_j$ will become a neutron-star-adjacent collection in some later step. Formally, for every $1\le j \le \ell$ with $j\neq i$ there is some step $t\ge 1$ such that $e_j \in \cross(t)$ for some $t\ge 1$, contradicting Claim~\ref{clm:nocrossmulti}. On the other hand, if $e_i$ intersects a hyperedge in an uninfected vertex, then we have that $e_i\in \cross(t)$ for some $t\ge 1$, again contradicting Claim~\ref{clm:nocrossmulti}. We conclude that no subset of $\{e_1,\ldots,e_{\ell}\}$ contains $r$ hyperedges which intersect only in $u$. 
	
	We have already established that each hyperedge in $\{e_1,\ldots,e_{\ell}\}$ has at most one vertex in $\linfec_f$. Together with the fact that other than $u$ any pair of hyperedges in $\{e_1,\ldots,e_{\ell}\}$ can intersect only in infected vertices, and no $r$ of them are vertex disjoint except for $u$ implies that $u$ has less than $r$ infected neighbours at the end of the process, contradicting our assumption that $u$ has at least $r$ infected neighbours. 
\end{proof}

\section{Subcritical regime}\label{sec:subcritical}

\subsection{Properties of the query-process}

In the following we examine the behaviour of the query-process when the underlying graph is the binomial random $k$-uniform hypergraph $H_k(n,p)$ with vertex set $V=[n]$ and hyperedge set $E$.
Recall that we only query collections, where none of the hyperedges have been exposed before. However the query may contain unexposed hyperedges which have been part of unsuccessfully queried collections. 
Intuitively, if a collection contains $k$-sets which were part of an unsuccessful query, then the probability that the current query is successful decreases.   
In the following lemma, we make this argument rigorous by providing an upper bound for the probability of a successful query, in a more general setting, which is independent of the result of all previous queries.

\begin{lemma}\label{lem:indepbound}
	Consider a process where collections of $k$-sets are subsequently queried, with a restriction that no queried collection may contain any $k$-set from a previously successfully queried collection, as in Algorithm 1. Then, conditional on the results of all previous queries, the probability that a collection $\coll$ of $k$-sets is successfully queried is at most $p^{|\coll|}$.
\end{lemma}

\begin{proof}
Denote by $E_1,\ldots,E_\ell$ the collections of $k$-sets which were unsuccessfully queried before $\coll$ was examined and by $S_1,\ldots,S_j$ the $k$-sets which were part of a successful query before $\coll$ was examined. Let $S=\bigcup_{i=1}^{j}S_i$ and define $E_i'=E_i\backslash S$ for $i=1,\ldots,\ell$. Since $E_i$ was unsuccessfully queried we have $E_i\setminus S\neq \emptyset$. 
The event $E_i\subseteq E$ and $E_i'\subseteq E$ are denoted by $\mathcal{E}_i$ and $\mathcal{E}_i'$ respectively and define $\overline{\mathcal{E}}=\bigcap_{i=1}^{\ell} \overline{\mathcal{E}_i}$ and $\overline{\mathcal{E}'}=\bigcap_{i=1}^{\ell} \overline{\mathcal{E}'_i}$. Finally let $\mathcal{S}$ be the event that $S\subseteq E$. 

Note that for every $1\le i \le \ell$ we have $ \mathcal{S}\cap \overline{\mathcal{E}_i}=\mathcal{S}\cap \overline{\mathcal{E}'_i}$ and thus $\mathcal{S}\cap\overline{\mathcal{E}}=\mathcal{S}\cap \overline{\mathcal{E}'}$.
In addition, $\coll\subseteq E$ and $\mathcal{S}$ are increasing events, while the event $\overline{\mathcal{E}'}$ is decreasing. Therefore, we have 
\begin{align*}
\mathbb{P}[\coll \subseteq E,\mathcal{S},\overline{\mathcal{E}}]
&=\mathbb{P}[\coll \subseteq E,\mathcal{S},\overline{\mathcal{E}'}]\\
&\leq \mathbb{P}[\coll \subseteq E,\mathcal{S}] \mathbb{P}[\overline{\mathcal{E}'}] && (\text{by FKG (Theorem~\ref{thm:FKG})}).
\end{align*}
Using the fact that $S=\bigcup_{i=1}^{j}S_i$ is disjoint of both the collection $\coll$ and every $E_i'$, we have 
\begin{align*}
\mathbb{P}[\coll \subseteq E,\mathcal{S}] \mathbb{P}[\overline{\mathcal{E}'}]
&= \mathbb{P}[\coll \subseteq E]\mathbb{P}[\mathcal{S}] \mathbb{P}[\overline{\mathcal{E}'}]&& (\text{because } \coll\cap S=\emptyset)\\
&= \mathbb{P}[\coll \subseteq E]\mathbb{P}[\mathcal{S},\overline{\mathcal{E}'}] && (\text{because } S\cap E_i'=\emptyset).
\end{align*}
Again using $\mathcal{S}\cap\overline{\mathcal{E}}=\mathcal{S}\cap \overline{\mathcal{E}'}$ we have
\begin{align*}
\mathbb{P}[\coll \subseteq E]\mathbb{P}[\mathcal{S},\overline{\mathcal{E}'}]&= \mathbb{P}[\coll \subseteq E]\mathbb{P}[\mathcal{S},\overline{\mathcal{E}}]= p^{|\coll|}\mathbb{P}[\mathcal{S},\overline{\mathcal{E}}].
\end{align*}
We conclude that
$$\mathbb{P}[\coll \subset E,\mathcal{S},\overline{\mathcal{E}}]\le p^{|\coll|}\mathbb{P}[\mathcal{S},\overline{\mathcal{E}}],$$
and result follows after dividing both sides by $\mathbb{P}[\mathcal{S},\overline{\mathcal{E}}]$.
\end{proof}

Using the previous lemma we estimate the distribution of the infected vertices.

\begin{corollary}\label{lem:stochdom}
Given $\cross(t+1), \multi(t+1), \overlap(t+1)$ and $\dom(t+1)$,  the number of vertices infected in step $t+1$ is stochastically dominated by

$$k^2\mathrm{Bin}(N(t+1),p)+k\mathrm{Bin}(H(t+1),p)+2k\mathrm{Bin}(W(t+1),p^2)+\eta \mathrm{Bin}(S(t+1),p^r),$$
where $\eta := (2k-3) \mathbbm 1_{r=2} + \mathbbm 1_{r\ge 3}$ is as defined in \eqref{eq:eta}.
\end{corollary}

\begin{proof}
Lemma~\ref{lem:indepbound} implies that the number of successful queries within the individual families  are stochastically dominated by $\mathrm{Bin}(N(t+1),p), \mathrm{Bin}(H(t+1),p), \mathrm{Bin}(W(t+1),p^2)$ and $\mathrm{Bin}(S(t+1),p^r)$, respectively. A successful query leads to different number of infected vertices depending on which family  the queried collection of $k$-sets was contained in. Since the hyperedges of $\infedge(t)$ overlap only in infected vertices we have that any $k$-set not contained in $\infedge(t)$ can share an uninfected vertex with at most $k-1$ hyperedges in $\infedge(t)$ and thus every neutron-star-adjacent collection (comprised of a single $k$-set by definition), causes at most $k^2$ vertices to become infected, after a successful query. When considering a heavily-infected collection (comprised of a single $k$-set by definition) only the vertices within the $k$-set become infected leading to at most $k$ infected vertices, in case of a successful query. Any widely-over-lapping collection (comprised of two $k$-sets by definition) contains at most $2k$ vertices, all of which are infected, after a successful query. Because $\eta=2k-3$ when $r=2$ and $\eta=1$ when $r\ge 3$, the successful query of a star collection leads to exactly $\eta$ infected vertices for any $r\ge 2$.
\end{proof}

\subsection{Bounding trajectory}

In the following we establish the expected trajectory for the number of infected vertices. 
Recall that $$a:=|\infec(0)|=|\linfec(0)|=(1-\varepsilon)a_c.$$ 
Throughout the remainder of this section we assume that $\delta>0$ satisfies $$1/(1+\delta)>(1-\varepsilon)^{r-1}.$$
For convenience of notation define $\linfec(-1)=\emptyset$. 
Based on the heuristic that \eqref{eq:forwardstep} holds, and under the assumption that in every step the number of infected vertices exceeds its expectation by at most a $(1+\delta)$ factor, we define the following bounding trajectory.

\begin{definition}[Bounding trajectory]\label{def:utrajectory}
We define the sequence $\nlinfecappr(t)$ by setting $b(-1)=0$, $b(0)=|\linfec(0)|$ and for each integer $t \geq 0$
\begin{align*}
 \nlinfecappr(t+1)&:=(1+\delta)\eta \frac{\nlinfecappr(t)^r}{r!}n\left[\binom{n}{k-2}p\right]^r+\nlinfecappr(0).
\end{align*}
Moreover, for every integer $t\ge -1$  we set 
$$\nlinfecrat(t):=\frac{\nlinfecappr(t)}{a^*},$$
where $a^*$ is as defined in \eqref{eq:scaling}. 
We call the function $\nlinfecrat$ the \emph{trajectory associated to the query-process}.
\end{definition} 

In the following we express $\nlinfecrat$ using a recursion.
 
\begin{claim}\label{claim:trajectoryRecursion}
For every integer $t \geq 0$ we  have
\begin{align*}
\nlinfecrat(t+1)= (1+\delta) \frac{\nlinfecrat(t)^{r}}{r}+ \nlinfecrat(0).
\end{align*} 
\end{claim}

\begin{proof}
By Definition~\ref{def:utrajectory} we have
\begin{align*}
 \nlinfecrat(t+1) 
&=(1+\delta)\eta \frac{\nlinfecrat(t)^{r}(a^*)^{r-1}}{r!}n\left[\binom{n}{k-2}p\right]^r+\nlinfecrat(0)\\
&\stackrel{\eqref{eq:scaling}}{=}(1+\delta)\eta \frac{\nlinfecrat(t)^{r}}{r!}\left(\frac{(r-1)!}{\eta n\left(\binom{n}{k-2}p\right)^r}\right)n\left[\binom{n}{k-2}p\right]^r+\nlinfecrat(0)\\
&=(1+\delta)\frac{\nlinfecrat(t)^{r}}{r}+\nlinfecrat(0).
\end{align*}
\end{proof}

Based on the heuristic in Section~\ref{ssec:heuristic} we expect $\beta(t)$ to also have a maximum. We will prove a slightly weaker statement, namely that $\beta\le x_0$ where $x_0$ is the smallest positive solution of the equation:
\begin{equation}\label{eq:boundx}
x_0=(1+\delta)\frac{x_0^{r}}{r}+\nlinfecrat(0).
\end{equation}
First we prove that a solution exists.

\begin{claim}\label{lem:boundx}
If $\beta(0)=(1-\varepsilon)(1-1/r)$, then $x_0$, defined in \eqref{eq:boundx}, satisfies  $$\nlinfecrat(0)<x_0<(1+\delta)^{-1/(r-1)}<1.$$
\end{claim}

\begin{proof}
First we observe that the function $h(x):=x-(1+\delta)x^r/r$ is strictly increasing and continuous on the interval $(0,(1+\delta)^{-1/(r-1)})$. Since $h(\nlinfecrat(0))\leq\nlinfecrat(0)$ and 
$$h\left(\left(\frac{1}{1+\delta}\right)^{1/(r-1)}\right)=\left(\frac{1}{1+\delta}\right)^{1/(r-1)}\left(1-\frac{1}{r}\right)>(1-\varepsilon)\left(1-\frac{1}{r}\right)=\nlinfecrat(0),$$
such a solution exists.
\end{proof}

Now we can show that $\nlinfecrat(t)<x_0$ for every $t\geq 0$, and thus by Claim~\ref{lem:boundx} we have $\nlinfecappr(t)\leq a^* x_0 <a^*$. 

\begin{lemma}\label{lem:ubound}
Assume that $\nlinfecrat(t)\leq (1-\xi)x_0$ for some $0<\xi<1$. Then there exists a $0<\xi'<1$ such that $\nlinfecrat(t+1)\leq (1-\xi')x_0$.
\end{lemma}

\begin{proof}
By Claim~\ref{claim:trajectoryRecursion} and $\nlinfecrat(t)\leq (1-\xi)x_0$ we have
\begin{align*}
 \nlinfecrat(t+1)&=(1+\delta)\frac{\nlinfecrat(t)^{r}}{r}+\nlinfecrat(0)\leq (1+\delta)\frac{(1-\xi)^{r}x_0^r}{r}+\nlinfecrat(0).
\end{align*}
Adding and subtracting $(1-\xi)^{r}\nlinfecrat(0)$ and using the definition of $x_0$ from \eqref{eq:boundx} gives
\begin{align*}
(1+\delta)\frac{(1-\xi)^{r}x_0^r}{r}+\nlinfecrat(0)
&=(1-\xi)^{r}\left((1+\delta)\frac{x_0^r}{r}+\nlinfecrat(0)\right)+(1-(1-\xi)^r)\nlinfecrat(0)\\
&\stackrel{\eqref{eq:boundx}}{=}(1-\xi)^{r}x_0+(1-(1-\xi)^r)\nlinfecrat(0).
\end{align*}
Using $\nlinfecrat(0)\le \nlinfecrat(t)\le (1-\xi)x_0$ and some simple algebra leads to 
\begin{align*}
(1-\xi)^{r}x_0+(1-(1-\xi)^r)\nlinfecrat(0)
&\leq (1-\xi)^{r}x_0+(1-(1-\xi)^r) (1-\xi)x_0\\
&=(1-\xi(1-(1-\xi)^r))x_0.
\end{align*}
Since $0<\xi<1$, we have $0<1-(1-\xi)^r<1$, and setting $\xi'=\xi(1-(1-\xi)^r)$ completes the proof.
\end{proof}

\subsection{Remaining below the trajectory}

We have established that the trajectory for the number of infected vertices remains below $a^*$. In the remainder of the section we show that, with sufficiently high probability, the number of infected vertices remains bounded from above by  the trajectory. 

We start by showing that as long as the number of infected vertices does not significantly exceed its expectation in each of the previous steps, this also holds with sufficiently high probability in the following step. Let $\mathcal{G}_t$ denote the event that for every $0<\tau\leq t$ we have $|\linfec(\tau)|-|\linfec(\tau-1)|\leq \nlinfecappr(\tau)-\nlinfecappr(\tau-1)$. 

\begin{lemma}\label{lem:upperstep}
For every $t\ge 1$, conditional on $\mathcal{G}_t$ 
we have 
$$|\linfec(t+1)|-|\linfec(t)|\leq \nlinfecappr(t+1)-\nlinfecappr(t)$$ with probability at least $1-\exp(-\delta^2(\nlinfecappr(t+1)-\nlinfecappr(t))/(2k^2))$.
\end{lemma}

\begin{proof}
We start by establishing an upper bound on the expected number of infected vertices.
By Claim~\ref{lem:boundx} and Lemma~\ref{lem:ubound} we have
\begin{equation}\label{eq:exnn}
\nlinfecappr(t)\binom{n}{k-2}p\le a^*\binom{n}{k-2}p=\Theta\left(\left(n\binom{n}{k-2}p\right)^{-1/(r-1)}\right)=o(1). 
\end{equation}
In addition, by Definition~\ref{def:utrajectory} and $\nlinfecappr(0)\le \nlinfecappr(t)\le \nlinfecappr(t+1)$ we have
\begin{align*}
\nlinfecappr(t+1)-\nlinfecappr(t)
&=(1+\delta)\eta \frac{\nlinfecappr(t)^{r}-\nlinfecappr(t-1)^{r}}{r!}n\left[\binom{n}{k-2}p\right]^r\\
&= \left(\nlinfecappr(t)-\nlinfecappr(t-1)\right)\Omega\left(b(0)^{r-1}n\left[\binom{n}{k-2}p\right]^r\right).\\
\end{align*}
Recall that $\nlinfecappr(0)=|\infec(0)|=\Omega(a^*)$. Thus by \eqref{eq:scaling}
\begin{align}
\nlinfecappr(t+1)-\nlinfecappr(t)
&=\left(\nlinfecappr(t)-\nlinfecappr(t-1)\right)\Omega\left((a^*)^{r-1}n\left[\binom{n}{k-2}p\right]^r\right)\nonumber\\
&=\Omega(\nlinfecappr(t)-\nlinfecappr(t-1)).\label{eq:sameorder}
\end{align}

In the following we consider the expected number of successful queries in each of the individual families.

We start with the family of neutron-star-adjacent collections. By Observation~\ref{obs:crosssize} we have
$$N(t+1)p\leq 2 r k(|\linfec(t)|-|\linfec(t-1)|) |\linfec(t)|\binom{n}{k-2}p.$$
By assumption we have $|\linfec(\tau)|-|\linfec(\tau-1)|\leq \nlinfecappr(\tau)-\nlinfecappr(\tau-1)$ for every $\tau\leq t$ implying that $|\linfec(t)|-|\linfec(t-1)|\le \nlinfecappr(t)-\nlinfecappr(t-1)$ and $\linfec(t)\le \nlinfecappr(t)$. Therefore
$$
N(t+1)p
=O\left((\nlinfecappr(t)-\nlinfecappr(t-1)) \nlinfecappr(t) \binom{n}{k-2}p\right)
\stackrel{\eqref{eq:exnn}}{=}o(\nlinfecappr(t)-\nlinfecappr(t-1))
\stackrel{\eqref{eq:sameorder}}{=}o(\nlinfecappr(t+1)-\nlinfecappr(t)).
$$
The following two calculations are done similarly.
When considering the number of heavily-infected collections we start from Observation~\ref{obs:multisize}, which implies
\begin{align*}
H(t+1)p&\leq (|\linfec(t)|-|\linfec(t-1)|)|\linfec(t)|\binom{n}{k-2}p
=o(\nlinfecappr(t+1)-\nlinfecappr(t)).
\end{align*}
In addition, if we consider widely-overlapping pairs of $k$-sets, then by Observation~\ref{obs:overlapsize} we have 
\begin{align*}
W(t+1)p^2&\leq |\linfec(t)|(|\linfec(t)|-|\linfec(t-1)|)n^2\binom{n}{k-3}^2p^2
=o\left(\nlinfecappr(t+1)-\nlinfecappr(t)\right).
\end{align*}
Finally we consider stars. Similarly as before, by assumption we have $|\linfec(\tau)|-|\linfec(\tau-1)|\leq \nlinfecappr(\tau)-\nlinfecappr(\tau-1)$ for every $\tau\leq t$ and thus $|\linfec(t)|-|\linfec(t-1)|\le \nlinfecappr(t)-\nlinfecappr(t-1)$, $\linfec(t)\le \nlinfecappr(t)$ and $\linfec(t-1)\le \nlinfecappr(t-1)$. Therefore
\begin{align*}
|\linfec(t)|^{r}-|\linfec(t-1)|^r
&=(|\linfec(t)|-|\linfec(t-1)|)\sum_{\rho=0}^{r-1}|\linfec(t)|^{\rho}|\linfec(t-1)|^{r-1-\rho}\\
&\le(\nlinfecappr(t)-\nlinfecappr(t-1))\sum_{\rho=0}^{r-1}\nlinfecappr(t)^{\rho}\nlinfecappr(t-1)^{r-1-\rho}\\
&=\nlinfecappr(t)^r-\nlinfecappr(t-1)^r.
\end{align*}
Together with Observation~\ref{obs:starsize} this implies
\begin{align*}
S(t+1)p^r&\leq \frac{|\linfec(t)|^{r}-|\linfec(t-1)|^r}{r!}n\left(\binom{n}{k-2}\right)^r p^r\\
&\leq \frac{\nlinfecappr(t)^r-\nlinfecappr(t-1)^r}{r!}n\left(\binom{n}{k-2}\right)^r p^r\\
&= \frac{\nlinfecappr(t+1)-\nlinfecappr(t)}{(1+\delta) \eta},
\end{align*}
where the last equality follows from Definition~\ref{def:utrajectory}.

By Corollary~\ref{lem:stochdom} we have that $|\linfec(t+1)|-|\linfec(t)|$ is stochastically dominated by
$$k^2\mathrm{Bin}(N(t+1),p)+k\mathrm{Bin}(H(t+1),p)+2k\mathrm{Bin}(W(t+1),p^2)+\eta \mathrm{Bin}(S(t+1),p^r).$$
Together with the previous calculations this implies $\mathbb{E}[|\linfec(t+1)|-|\linfec(t)|]\le(1+o(1))(\nlinfecappr(t+1)-\nlinfecappr(t))/(1+\delta)$
and applying Theorem~\ref{thm:indconc} gives us
\begin{align*}
\mathbb{P}[(|\linfec(t+1)|-|\linfec(t)|)\geq \nlinfecappr(t+1)-\nlinfecappr(t)]&\leq \exp \left(-\frac{(1+o(1))\delta^2(\nlinfecappr(t+1)-\nlinfecappr(t))^2}{2((\nlinfecappr(t+1)-\nlinfecappr(t))+k^2(\nlinfecappr(t+1)-\nlinfecappr(t))/3)}\right)\\
&\leq \exp \left(-\frac{\delta^2(\nlinfecappr(t+1)-\nlinfecappr(t))}{2k^2}\right),
\end{align*}
as required.
\end{proof}

The previous Lemma provides sufficiently high probabilities when the expected number of vertices infected in a given step is large, that is, $\Omega(a^*)$. However, the expected number of infected vertices decreases gradually and thus after a while it will be impossible to establish tight concentration on a step by step basis, and we will require a different method.
In particular we approximate the newly infected vertices with a Galton-Watson branching process and show that it dies out quickly using the Dwass identity. We start by providing an upper bound on the size of a Galton-Watson process.

\begin{lemma}\label{lem:branchsize}
Consider a Galton-Watson process where the number of offspring is the sum of weighted Bernoulli random variables, that is, $\sum_{i=1}^k w_i Be(p_i)$, such that $\mu:=\sum_{i=1}^k w_ip_i< 1$ for $w_i\in \mathbb R^+$ and $p_i\in (0,1)$ for each $1\leq i\leq k$. Set $M=\max_{i=1,\ldots,k}w_i$ and denote by $Z_j$ the number of individuals in the $j$-th generation and by $Z$ the total number of individuals in the process.  Then for any $\chi\ge \mu/(1-\mu)$ and $\ell\in \mathbb N$ we have
$$\mathbb{P}\left[Z>(1+\chi)\ell \mid Z_0=\ell\right]\le \frac{\exp\left(-(1-(1+\chi)^{-1}-\mu)^2(1+\chi)\ell/(3M)\right)}{1-\exp\left(-(1-(1+\chi)^{-1}-\mu)^2(1+\chi)/(3M)\right)}.$$

 \end{lemma}

\begin{proof}
By the Dwass identity (Theorem~\ref{thm:Dwass}) we have that  for any $m>(1+\chi)\ell$,
$$\mathbb{P}\left[Z=m\mid Z_0=\ell\right]
=\frac{\ell}{m} \mathbb{P}\left[Z_1=m-\ell\mid Z_0= m\right] < \mathbb{P}\left[Z_1=m-\ell\mid Z_0= m\right],$$
where the last inequality is due to  $\frac{\ell}{m}<\frac{1}{1+\chi}\le 1 -\mu < 1$, and thus
$$\mathbb{P}\left[Z> (1+\chi)\ell \mid Z_0=\ell\right]\le \sum_{m> (1+\chi)\ell} \mathbb{P}\left[Z_1=m-\ell\mid Z_0= m\right].$$
Using $m >(1+\chi)\ell$ and that the expected number of offspring in the branching process is $\mu$ gives
\begin{align*}
\mathbb{P}\left[Z_1=m-\ell\mid Z_0= m\right]
&\le \mathbb{P}\left[Z_1\ge m-\ell\mid Z_0= m\right]\\
&= \mathbb{P}\left[Z_1-\mathbb{E}\left[Z_1\mid Z_0=m\right]\ge m-\ell-\mathbb{E}\left[Z_1\mid Z_0=m\right]\mid Z_0= m\right]\\
&\le \mathbb{P}\left[Z_1-\mathbb{E}\left[Z_1\mid Z_0=m\right]\ge m-\frac{1}{1+\chi}m-\mu m\mid Z_0= m\right].
\end{align*}

Since $\chi\ge \mu/(1-\mu)$ we have
$$1-\frac{1}{1+\chi}-\mu\ge 0$$
and thus Theorem~\ref{thm:indconc} implies
\begin{align*}
\mathbb{P}\left[Z_1=m-\ell\mid Z_0= m\right]
&\le \exp\left(-\frac{(1-(1+\chi)^{-1}-\mu)^2}{3M}m\right),
\end{align*}
where we use $\mathrm{Var}(Z_1\mid Z_0=m)\le M \mathbb{E}\left[Z_1\mid Z_0 =m\right]\le M m.$

\end{proof}

Recall that $\mathcal{G}_t$ is the event that for every $0<\tau\leq t$ we have $|\linfec(\tau)|-|\linfec(\tau-1)|\leq \nlinfecappr(\tau)-\nlinfecappr(\tau-1)$. 
\begin{lemma}\label{lem:branching}
Set $\chi=4/\delta$. Let $t_0$ be such that $\nlinfecappr(t_0-1)=(1-\xi) x_0 a^*$ where 
\begin{equation}\label{eq:xicond}
\left[\left(\frac{1+\delta}{1+\delta/2}\right)^{1/(r-1)}-1\right]\frac{1}{\chi}> \xi.
\end{equation} 
Then
$$ \mathbb{P}\left[|\linfec_f| \geq \frac{a^*}{(1+\delta/2)^{1/(r-1)}}\mid \mathcal{G}_{t_0}\right]\leq \exp\Big(-\Omega\big(\nlinfecappr(t_0)-\nlinfecappr(t_0-1)\big)\Big).$$
\end{lemma}

\begin{proof}
By Lemma~\ref{lem:ubound} and Definition~\ref{def:utrajectory} we have $\nlinfecappr(t)\le x_0 a^*$ for every $t\ge 0$, which together with $\nlinfecappr(t_0-1)=(1-\xi) x_0 a^*$ implies $\nlinfecappr(t_0)-\nlinfecappr(t_0-1)\leq \xi x_0 a^*$, thus 	
\begin{equation}\label{eq:aprimecalc1}
\nlinfecappr(t_0-1)+(1+\chi)(\nlinfecappr(t_0)-\nlinfecappr(t_0-1))\leq (1-\xi)x_0 a^*+(1+\chi)\xi x_0 a^*
=(1+\xi \chi) x_0 a^*.
\end{equation}
Using the upper bound on $x_0$ from Claim~\ref{lem:boundx} and \eqref{eq:xicond} gives
\begin{align*}
(1+\xi \chi) x_0 &\stackrel{Clm~\ref{lem:boundx}}{\leq} \frac{1+\xi \chi}{(1+\delta)^{1/(r-1)}}
\stackrel{\eqref{eq:xicond}}{\leq} \frac{1}{(1+\delta/2)^{r-1}}.
\end{align*}
Together with \eqref{eq:aprimecalc1} this implies
\begin{equation}\label{eq:aprimecalc}
\nlinfecappr(t_0-1)+(1+\chi)(\nlinfecappr(t_0)-\nlinfecappr(t_0-1)) \le \frac{a^*}{(1+\delta/2)^{r-1}}=:a'.
\end{equation}

Note that if $\mathcal{G}_{t_0}$ holds then $\linfec(t_0-1)\le \nlinfecappr(t_0-1)$, 
therefore an upper bound on the probability that the total number of vertices infected after step $t_0-1$ is at least $(1+\chi)(\nlinfecappr(t_0)-\nlinfecappr(t_0-1))$ is an upper bound on the probability that $|\linfec_f|\ge a'$.

{\bf Forest construction.} We construct a forest from the vertices in $\linfec_f\backslash \linfec(t_0-1)$. The roots of the individual trees in the forest are the vertices in $\linfec(t_0)\setminus \linfec(t_0-1)$. For $t\ge t_0$ we will assign each vertex infected in step $t+1$, that is, each vertex in the set $\linfec(t+1)\backslash\linfec(t)$, to  at least one vertex infected in the previous step, that is, a vertex in $\linfec(t)\backslash \linfec(t-1)$, and then discard some edges to transform this into a tree. 

In order to achieve this we will consider subsets of $\dom(t+1),\overlap(t+1),\multi(t+1)$ and $\cross(t+1)$ associated to one vertex $u\in \linfec(t)\setminus\linfec(t-1)$ as follows:
\begin{itemize}
	\item $\dom_u(t+1)\subseteq \dom(t+1)$, where one of the $k$-sets in the collection contains $u$;
	\item $\overlap_u(t)\subseteq \overlap(t)$, where one of the $k$-sets in the pair contains $u$;
	\item $\multi_u(t+1)\subseteq \multi(t)$, where the $k$-set contains $u$;
	\item $\cross_u(t+1)\subseteq \cross(t+1)$ be the set of $k$-sets satisfying one of the following conditions:
	\begin{itemize}
		\item $K\in \cross_1(t+1)$ and $u\in K$;
		\item $K\in \cross_2(t+1)$ and there exists $e\in \infedge(t)\setminus \infedge(t-1)$ such that $K\cap e \cap \linfec(t)^C\neq \emptyset$ and $u\in e$.
	\end{itemize}
\end{itemize}

We say that $u\in \linfec(t)\setminus \linfec(t-1)$ is a potential parent of $v\in \linfec(t+1)\setminus \linfec(t)$ if one of the following two options hold. First, if there exists a successfully queried collection of $k$-sets in $\dom_u(t+1),\overlap_u(t+1)$, or $\multi_u(t+1)$, where one of the hyperedges in the collection contains $v$. Second, if there exists a successfully queried $k$-set $K\in \cross_u(t+1)$ such that either $v\in K$ or $v\in e$ for some hyperedge $e$ which intersects $K$ in an uninfected vertex, that is $K\cap e \cap \linfec(t)^C\neq\emptyset$.  This is visualised in Figure~\ref{fig:forestconstruction}.

\begin{figure}[ht]
	\begin{center}
		\begin{tikzpicture}[scale=0.7]
			
			\node (S0) at (-4,0) [circle, draw, fill=red, inner sep=2pt, label=right:$v$] {};
			
			\node (S1) at (-5,-1) [circle, draw,  inner sep=2pt] {};
			\node (S2) at (-6,-2) [circle, draw, inner sep=2pt] {};
			\node (S3) at (-7,-3) [circle, fill=blue, inner sep=2pt, label=left:$u$] {};
			
			\node (S4) at (-5,0) [circle, draw, inner sep=2pt] {};
			\node (S5) at (-6,0) [circle, draw, inner sep=2pt] {};
			\node (S6) at (-7,0) [circle, fill=black, inner sep=2pt] {};
			
			\node (S7) at (-5,1) [circle, draw, inner sep=2pt] {};
			\node (S8) at (-6,2) [circle, draw, inner sep=2pt] {};
			\node (S9) at (-7,3) [circle, fill=black, inner sep=2pt] {};
			
			\draw[black,very thick]\convexpath{S0,S3}{0.27cm};
			\draw[black,very thick]\convexpath{S0,S6}{0.27cm};
			\draw[black,very thick]\convexpath{S0,S9}{0.27cm};
			
			\node at (-6,-7.5){(a) stars};
			
			\node (X00) at (-6,-5) [circle, fill=blue, inner sep=2pt, label=$u$] {};
			\node (X01) at (-6,-6) [circle, draw, fill=red, inner sep=2pt, label=below:$v$] {};
			\draw (X00)--(X01);

			\node (O0) at (1,0) [circle,draw, fill=red, inner sep=2pt, label=right:$v_1$] {};
			\node (O1) at (0,0) [circle,draw, fill=red, inner sep=2pt, label=right:$v_2$] {};
			
			\node (O2) at (-1,-1) [circle,draw, fill=red, inner sep=2pt, label=right:$v_3$] {};
			\node (O3) at (-2,-2) [circle,fill=blue, inner sep=2pt, label=left:$u$] {};
			
			\node (O4) at (-1,1) [circle,draw, fill=red, inner sep=2pt, label=right:$v_4$] {};
			\node (O5) at (-2,2) [circle,fill=black, inner sep=2pt] {};
			
			\draw[black,very thick]\convexpath{O3,O1,O0}{0.27cm};
			\draw[black,very thick]\convexpath{O5,O0,O1}{0.27cm};
			
			\node at (-0.9,-7.5)[text width=3.5cm]{(b) widely-overlapping};
			
			\node (X10) at (-0.5,-5) [circle, fill=blue, inner sep=2pt, label=$u$] {};
			\node (X11) at (-2,-6) [circle, draw, fill=red, inner sep=2pt, label=below:$v_1$] {};
			\node (X12) at (-1,-6) [circle, draw, fill=red, inner sep=2pt, label=below:$v_2$] {};
			\node (X13) at (0,-6) [circle, draw, fill=red, inner sep=2pt, label=below:$v_3$] {};
			\node (X14) at (1,-6) [circle, draw, fill=red, inner sep=2pt, label=below:$v_4$] {};
			\draw (X10)--(X11);
			\draw (X10)--(X12);
			\draw (X10)--(X13);
			\draw (X10)--(X14);			
			
			\node (M0) at (6,0) [circle,draw, fill=red, inner sep=2pt,label=$v_2$] {};
			\node (M1) at (5,0) [circle, draw, fill=red, inner sep=2pt, label=$v_1$] {};
			\node (M2) at (4,0) [circle, fill=blue, inner sep=2pt,label=$u_2$] {};
			\node (M3) at (3,0) [circle, fill=blue, inner sep=2pt, label=$u_1$] {};
			
			\draw[black,very thick]\convexpath{M0,M3}{0.27cm};
			
			\node at (4.5,-7.5)[text width=3cm]{(c) heavily-infected};
			
			\node (X20) at (4,-5) [circle, fill=blue, inner sep=2pt, label=$u_1$] {};
			\node (X21) at (5,-5) [circle, fill=blue, inner sep=2pt, label=$u_2$] {};
			\node (X22) at (4,-6) [circle, draw, fill=red, inner sep=2pt, label=below:$v_1$] {};
			\node (X23) at (5,-6) [circle, draw, fill=red, inner sep=2pt, label=below:$v_2$] {};
			\draw (X20)--(X22);
			\draw (X20)--(X23);
			\draw (X21)--(X22);
			\draw (X21)--(X23);

			\node (L0) at (11,0) [circle,fill=blue, inner sep=2pt,label=right:$u_2$] {};
			
			\node (L1) at (10,-1) [circle, draw, inner sep=2pt] {};
			\node (L2) at (9,-2) [circle, draw, inner sep=2pt] {};
			\node (L3) at (8,-3) [circle, fill=black, inner sep=2pt] {};
			
			\node (L4) at (10,0) [circle, draw, inner sep=2pt] {};
			\node (L5) at (9,0) [circle, draw, inner sep=2pt] {};
			\node (L6) at (8,0) [circle, fill=black, inner sep=2pt] {};
			
			\node (L7) at (10,1) [circle, draw, fill=red, inner sep=2pt, label=below right:$v_4$] {};
			\node (L8) at (9,2) [circle, draw, fill=red, inner sep=2pt] {};
			\node (L9) at (8,3) [circle, fill=black, inner sep=2pt] {};
			
			\node (L10) at (8,2) [circle, fill=blue, inner sep=2pt, label=left:$u_1$] {};
			\node (L11) at (9,2) [circle, draw, fill=red, inner sep=2pt,label=above right:$v_1$] {};
			\node (L13) at (10,2) [circle, draw, fill=red, inner sep=2pt,label=above right:$v_2$] {};
			\node (L12) at (11,2) [circle, draw, fill=red, inner sep=2pt,label=above right:$v_3$] {};
			
			\draw[black,very thick]\convexpath{L0,L3}{0.27cm};
			\draw[black,very thick]\convexpath{L0,L6}{0.27cm};
			\draw[black,very thick]\convexpath{L0,L9}{0.27cm};
			\draw[black,very thick]\convexpath{L10,L12}{0.27cm};
			
			\node at (9.5,-7.5)[text width=3cm] {(d) neutron-star adjacent};
			
			\node (X30) at (9,-5) [circle, fill=blue, inner sep=2pt, label=$u_1$] {};
			\node (X31) at (10,-5) [circle, fill=blue, inner sep=2pt, label=$u_2$] {};
			\node (X32) at (8,-6) [circle, draw, fill=red, inner sep=2pt, label=below:$v_1$] {};
			\node (X33) at (9,-6) [circle, draw, fill=red, inner sep=2pt, label=below:$v_2$] {};
			\node (X34) at (10,-6) [circle, draw, fill=red, inner sep=2pt, label=below:$v_3$] {};
			\node (X35) at (11,-6) [circle, draw, fill=red, inner sep=2pt, label=below:$v_4$] {};
			\draw (X30)--(X32);
			\draw (X30)--(X33);
			\draw (X30)--(X34);			
			\draw (X30)--(X35);						
			\draw (X31)--(X32);
			\draw (X31)--(X33);			
			\draw (X31)--(X34);
			\draw (X31)--(X35);			
			
		\end{tikzpicture}
		\caption{Potential parents.\\ Empty circles represent uninfected vertices, while the filled circles represent infected vertices. Red circles were infected in step $t+1$, blue in step $t$ and black in the first $t-1$ steps.}
		\label{fig:forestconstruction}
	\end{center}
\end{figure}

Note that $\dom(t+1)=\bigcup_{u\in \linfec(t)\setminus \linfec(t-1)} \dom_u(t+1)$ and the analogous statements hold for $\overlap(t+1),\multi(t+1)$ and $\cross(t+1)$.
Therefore, every vertex $v\in \linfec(t+1)\backslash\linfec(t)$ has at least one potential parent. 
As we would like this relationship to be unique for every vertex $v$, we assign a unique parent, say the smallest vertex $u$, and thus for any $u_1,u_2$ the set of offspring of $u_1$ and $u_2$ are disjoint.

This construction is reminiscent of a branching process, except the offspring distribution of the individual vertices depends on the current state, in particular on the size of the sets $\dom_u(t+1),\overlap_u(t+1),\multi_u(t+1)$, and $\cross_u(t+1)$. In the following we will remove this dependence.

For $i\ge 0$ let $M_i=|\linfec(t_0+i)\setminus \linfec(t_0+i-1)|$.
Consider a stopping condition when $\sum_{i=0}^j \bpo_i$ becomes larger than $(1+\chi)(\nlinfecappr(t_0)-\nlinfecappr(t_0-1))$, i.e.,\ we define the random variable
$$
\bps_i =
\left\{
\begin{array}{ll}
	\bpo_i  & \mbox{if } \sum_{j=0}^{i-1}\bps_{j} <  (1+\chi)(\nlinfecappr(t_0)-\nlinfecappr(t_0-1)),\\
	0 & \mbox{otherwise},
\end{array}
\right.
$$
and similarly as before let $\bps=\sum_{i=0}^{\infty}\bps_i$. Clearly 
$$\mathbb{P}[\bpo\ge (1+\chi)(\nlinfecappr(t_0)-\nlinfecappr(t_0-1))\mid \mathcal{G}_t]=\mathbb{P}[\bps\ge (1+\chi)(\nlinfecappr(t_0)-\nlinfecappr(t_0-1))\mid \mathcal{G}_t].$$

Now consider the offspring distribution of a vertex in the $i$-th generation of $\bps$. 
Recall that the successful query of a star collection leads to $\eta$ infected vertices, the successful query of a widely-overlapping collection infects at most $2k$ infected vertices, the successful query of a heavily-infected collection results in at most $k$ vertices becoming infected and the successful query of a neutron-star-adjacent collection gives at most $k^2$ infected vertices. 
Therefore, by mimicking the proof of Corollary~\ref{lem:stochdom}, Lemma~\ref{lem:indepbound} implies that the number of offspring is stochastically dominated by
$$k^2\mathrm{Bin}(|\cross_u(t+1)|,p)+k\mathrm{Bin}(|\multi_u(t+1)|,p)+2k\mathrm{Bin}(|\overlap_u(t+1)|,p^2)+\eta \mathrm{Bin}(|\dom_u(t+1)|,p^r).$$

If the process has not stopped yet, then we have $\nlinfecappr(t_0-1)+\sum_{j=0}^i M_j' \le a'$ and  
following the calculations in Observations~\ref{obs:starsize}, \ref{obs:overlapsize}, \ref{obs:multisize} and \ref{obs:crosssize} we deduce that 

\begin{enumerate}[label=(\alph*)]
\item $|\cross_u(t+1)|\le 2rk a' \binom{n}{k-2}$;
\item $|\multi_u(t+1)|\le a'\binom{n}{k-2}$;
\item $|\overlap_u(t+1)|\le n^2 a'\binom{n}{k-3}^2$;
\item $|\dom_u(t+1)|\le  n \binom{a'}{r-1}\binom{n}{k-2}^r$.
\end{enumerate}
Thus we deduce that the offspring distribution is stochastically dominated by
\begin{equation}\label{eq:offspringdistr}
k^2 \mathrm{Bin}\left(3rk a' \binom{n}{k-2},p\right)+2k\mathrm{Bin}\left( n^2 a'\binom{n}{k-3}^2,p^2\right)+\eta \mathrm{Bin}\left( n \binom{a'}{r-1}\binom{n}{k-2}^r,p^r\right),
\end{equation}
and note that this also holds after the stopping time.

{\bf Coupling with Galton-Watson process.}
Denote by $\bpd$ the Galton-Watson process with $\nlinfecappr(t_0)-\nlinfecappr(t_0-1)$ roots  where the number of children of every vertex follows \eqref{eq:offspringdistr}.
Since $\mathcal{G}_{t_0}$ holds, the number of roots in $\bps$ is at most the number of roots in $\bpd$ and the offspring distribution for $\bpd$ stochastically dominates the offspring distribution for any vertex in $\bps$ implying
$$\mathbb{P}[\bps\ge (1+\chi)(\nlinfecappr(t_0)-\nlinfecappr(t_0-1))\mid \mathcal{G}_t]\le \mathbb{P}[\bpd\ge (1+\chi)(\nlinfecappr(t_0)-\nlinfecappr(t_0-1))].$$

Note that the expected number of offspring of an individual in $Z$ is
$$3rk^3 a' \binom{n}{k-2}p+2kn^2 a'\binom{n}{k-3}^2p^2+ \eta n\binom{a'}{r-1}\binom{n}{k-2}^r p^r \stackrel{\eqref{eq:scaling},\eqref{eq:aprimecalc}}{\leq} \frac{1}{1+\delta/3},$$
and the maximal weight is $k^2$.

Note that 
$$\chi=\frac{4}{\delta}> \frac{3}{\delta}= \frac{(1+\delta/3)^{-1}}{1-(1+\delta/3)^{-1}}.$$

Therefore by Lemma~\ref{lem:branchsize} we have
$$\mathbb{P}[\bpd\ge (1+\chi)(\nlinfecappr(t_0)-\nlinfecappr(t_0-1))]=\exp\left(-\Omega\left(\nlinfecappr(t_0)-\nlinfecappr(t_0-1)\right)\right),$$
and the statement follows.
\end{proof}

\subsection{Proof of Theorem~\ref{thm:hypergraph}\,\ref{case:sub}}
Using Lemma~\ref{lem:upperstep} we can bound the number of vertices infected in a given step. A repeated application of this lemma allows us to show that the number of vertices which become infected in a step decreases sufficiently over a bounded number of steps. Then we complete the proof by applying Lemma~\ref{lem:branching}. In this section we focus on the interaction of these two lemmas, and in particular we determine the number of steps required for the conditions of Lemma~\ref{lem:branching} to be met and calculate the relevant probabilities using the union bound.     

\begin{proof}[Proof of Theorem~\ref{thm:hypergraph}~\ref{case:sub}]

Set $\chi=4/\delta$. For any fixed $\xi'>0$ satisfying $\left[\left(\frac{1+\delta}{1+\delta/2}\right)^{1/(r-1)}-1\right]\frac{1}{\chi}> \xi'$ let $t_0(\xi')$ be the smallest value of $t$ such that $\beta(t-1)\geq (1-\xi')x_0$. First we show that $t(\xi')=O(1)$ and $\beta(\tau+1)-\beta(\tau)=\Omega(1)$ for every $\tau\leq t_0(\xi')$.

Assume that $\nlinfecrat(\tau)\leq (1-\xi')x_0$. Claim~\ref{lem:boundx} implies that $x_0\leq (1+\delta)^{-1/(r-1)}$ and note that the function $(1+\delta)y^r/r-y$ is monotone decreasing on the interval $[0,x_0]$. Therefore by Definition~\ref{def:utrajectory} 
\begin{align}
\nlinfecrat(\tau+1)-\nlinfecrat(\tau)&=(1+\delta)\frac{\nlinfecrat(\tau)^r}{r}+\nlinfecrat(0)-\nlinfecrat(\tau)\nonumber\\
&\geq  (1+\delta)\frac{(1-\xi')^r x_0^r}{r}+\nlinfecrat(0)-(1-\xi')x_0\nonumber\\
&\stackrel{\eqref{eq:boundx}}{=} (1-\xi')^r (x_0-\nlinfecrat(0))+\nlinfecrat(0)-(1-\xi')x_0\nonumber\\
&= (1-(1-\xi')^r)\nlinfecrat(0)-(1-\xi')(1-(1-\xi')^{r-1})x_0. \label{eq:lowerbeta1}
\end{align}
Next we bound $\nlinfecrat(0)$ from bellow. By \eqref{eq:boundx} we have 
\begin{align}
\beta(0)=x_0-(1+\delta)\frac{x_0^r}{r}=x_0\left(1-(1+\delta)\frac{x_0^{r-1}}{r}\right)\stackrel{Clm.\ref{lem:boundx}}{\geq} x_0\left(1-\frac{1}{r}\right). \label{eq:lowerbeta2}
\end{align}
Together \eqref{eq:lowerbeta1} and \eqref{eq:lowerbeta2} imply
\begin{align*}
\nlinfecrat(\tau+1)-\nlinfecrat(\tau)&\geq (1-(1-\xi')^r)\nlinfecrat(0)-(1-\xi')(1-(1-\xi')^{r-1})x_0\\
&\geq (1-(1-\xi')^r)\left(1-\frac{1}{r}\right)x_0-(1-\xi')(1-(1-\xi')^{r-1})x_0.
\end{align*}
Since $0<\xi'<1$ and $r\ge 2$, using the summation formula for the geometric series gives
\begin{align*}
	\frac{1-(1-\xi')^r}{r}\le \xi'\frac{1+(r-1)(1-\xi')}{r}=\xi'-\frac{r-1}{r}(\xi')^2,
\end{align*}
and we also have
\begin{align*}
(1-(1-\xi')^r)-(1-\xi')(1-(1-\xi')^{r-1})=\xi'.
\end{align*}
 
Together these imply
\begin{align*}
\nlinfecrat(\tau+1)-\nlinfecrat(\tau)&\geq \frac{r-1}{r}(\xi')^2 x_0.
\end{align*}

Therefore in every step before $t_0(\xi')$ the value of $\nlinfecrat(t)$ increases by at least $(r-1)(\xi')^2 x_0/r >0$ implying $t_0(\xi')=O(1)$ and $\beta(\tau+1)-\beta(\tau)=\Omega(1)$ for every $\tau\leq t_0(\xi')$. 

Now fix $\xi'$, set $t_0=t_0(\xi')$ and define $\xi$ implicitly by $\nlinfecrat(t_0-1)=(1-\xi)x_0$.
Lemma~\ref{lem:upperstep} implies that for every $0<t\leq t_0$ we have $|\linfec(t)|-|\linfec(t-1)|\leq \nlinfecappr(t)-\nlinfecappr(t-1)$ with probability $1-\exp(-\Omega(a^*))$ and Lemma~\ref{lem:branching} implies that with probability $1-\exp(-\Omega(a^*))$ the total number of infected vertices is at most $a^*$, completing the proof.

\end{proof}

\section{supercritical regime}\label{sec:supercritical}

In this section we turn our attention to the supercritical regime of Theorem~\ref{thm:hypergraph}. Throughout this section  we assume that  a constant $0<\varepsilon<1$ is given and the size of the initial infection set satisfies
	$$a:=|\infec(0)|=(1+\varepsilon)a_c.$$  
Similarly to the previous section we define a new process, called {\em mild $r$-bootstrap} percolation, which will infect  overall fewer vertices than the original $r$-bootstrap percolation. In particular, in each step we activate the right amount of vertices from the set of infected vertices, which have never been active before, and only expose the hyperedges which contain exactly one vertex which was activated in this step and no additional infected vertices. When $r\ge 3$, for a given vertex we count the number of its active neighbours within the hypergraph spanned by the hyperedges exposed so far and infect it, when the number of active neighbours reaches $r$. When $r=2$ the situation is slightly different. If a vertex has two active neighbours within the hyperedges exposed so far, then obviously it becomes infected.  
Following the rules of $r$-bootstrap percolation, this implies that every vertex within its neighbourhood, with respect to the hyperedges exposed so far, would also become infected, as every hyperedge contains an active vertex which is infected. Therefore in each step we add every vertex which has two active neighbours, and every vertex in its neighbourhood, within 
 the hyperedges exposed until this step.

The majority of the proof for the supercritical regime of Theorem~\ref{thm:hypergraph} considers the case when the number of infected vertices is small, that is, $o((n^{k-2}p)^{-1})$. 
In order to establish the probability that a given vertex becomes infected, we will not only need to track the number of infected vertices, but also the number of vertices which have at least $i$ active neighbours for $0\leq i \leq r$.
We start by establishing some properties of mild $r$-bootstrap percolation within one step. Then we establish for each $i$ the trajectory for the number of vertices with at least $i$ active neighbours and do similarly for the number of infected vertices. In addition, we prove that the typical behaviour of the process in fact follows these trajectories as long as the number of infected vertices remains small (Lemma~\ref{lem:uindstep}), which guarantees a steady increase in the number of infected vertices, until roughly $(n^{k-2}p)^{-1}$ vertices become infected. After this point the process finishes in 2 steps (Lemmas~\ref{lem:penultimatestep} and \ref{lem:finalstep}).

\subsection{Lower coupling via mild $r$-bootstrap percolation}

In mild $r$-bootstrap percolation the set of infected vertices is denoted by $\uinfec(t)$ and set $\uinfec(0)=\infec(0)$. Let  $\uinfec_f$ denote the set of vertices which become infected during mild $r$-bootstrap percolation.
We also introduce $\act(t)\subseteq \uinfec(t)$ which contains the vertices which have been already activated and set $\act(0)=\emptyset$. 
In addition, we require the sets $\uinfec_i(t)\subseteq V\setminus \uinfec(0)$ for $0\le i \le r$, which  contain the subset of the vertices with at least $i$ active neigbhours. Finally we let $\infedge(t)$ denote the set of hyperedges exposed until the end of step $t$, with $\infedge(0)=\emptyset$.

In step $t+1$ of the process we select a sufficiently large set $\act'(t+1)\subseteq \uinfec(t)\setminus \act(t)$, where the exact size of this set is defined in \eqref{eq:activesize}, and expose every hyperedge which contains exactly one vertex in $\act'(t+1)$ and no vertex in $\uinfec(t)$. Denote this set of hyperedges by $\infedge'(t+1)$. First set $\act(t+1)=\act(t)\cup \act'(t+1)$ and $\infedge(t+1)=\infedge(t)\cup \infedge'(t+1)$. 
In addition let $\uinfec_i(t+1)$ contain every vertex in $\uinfec_i(t)$ and every vertex found in $\uinfec_j(t)$ which has at least $i-j$ neighbours in $\act'(t+1)$ in the hypergraph spanned by $\infedge'(t+1)$. 
Finally we define the set of infected vertices. When $r\ge 3$, we simply have $\uinfec(t+1)=\uinfec(t)\cup \uinfec_r(t+1)$. On the other hand, when $r=2$, the set $\uinfec(t+1)$ contains every vertex in $\uinfec(t)$ and every vertex in the closed neighbourhood of $\uinfec_2(t+1)\setminus \uinfec_2(t)$ within the hypergraph spanned by $\infedge(t+1)$.

Note that changing the step in which a vertex becomes infected, has no affect on the final set of infected vertices. 
Moreover, infection will spread only to a smaller set if some infected vertices remain inactive.
Together with the fact that in mild $r$-bootstrap percolation only a subset of the hyperedges are exposed when determining whether a vertex becomes infected, we have $$\uinfec_f\subseteq \infec_f.$$ 

\subsection{Properties of mild $r$-bootstrap percolation}
The aim of this section is to establish some properties of mild $r$-bootstrap percolation which will lead to lower bounds for the number of vertices added to $\uinfec_{i}(t)$ and $\uinfec(t)$ in a single step, during the early stages of the process.

Roughly speaking, when we add every vertex in a hyperedge to a set, that is, for $\uinfec_1(t)$ and $\uinfec(t)$ when $r=2$, there is a strong correlation between two vertices being added to the set. 
In the remaining cases, namely for $\uinfec_i(t)$ when $2\le i \le r$ and $\uinfec(t)$ for $r\ge 3$, this is no longer the case, as in these cases two vertices will rarely be found in the same hyperedge. We require separate arguments for these two cases, which will be discussed in the remainder of this section, starting with the latter.

In our approach we will use Lemma~\ref{lem:negcorr} (for dependent random variables) to provide an upper bound on the lower tail for the number of vertices added to $C_i(t)$ in a given step.  
As Lemma~\ref{lem:negcorr} provides an upper bound on the upper tail of a random variable, we will use it to derive an upper bound on the number of vertices not added to $C_i(t)$ in a given step. 
In the following lemma we show that the associated random variables satisfy the condition of Lemma~\ref{lem:negcorr}.

\begin{lemma}\label{lem:negcorrappl}
Consider $H_k(n,p)$ with vertex set $V=[n]$ and hyperedge set $E$. Let $U,V'\subseteq V$ be two disjoint subsets with $|V'|=o(n)$ while  $U$ satisfies both $|U|=\Omega(a_c)$ and $|U|=o((n^{k-2}p)^{-1})$. Denote by $E'\subseteq E$ the set of hyperedges which contain exactly one vertex in $U$ and no vertices in $V'$. Consider $W\subseteq V\setminus (U\cup V')$ and an arbitrary function $g:W\to \{1,\ldots,r\}$ such that $|\{u:g(u)=1\}|=o(n)$. 
For $u\in W$ let $Z_u$ be the indicator random variable for the event that $u$ has less then $g(u)$ neighbours in $U$, within the graph spanned by $E'$.
Then we have for every $I\subseteq W$
$$\mathbb{E}\left[\prod_{i\in I} Z_i\right]\leq \mathbb{E}\left[\prod_{i\in I} \widehat{Z}_i\right],$$
where $\widehat{Z}_u$ for $u\in W$ forms a set of independent Bernoulli random variables satisfying 
$$\mathbb{P}[\widehat{Z}_u=1]=1-(1-o(1))\frac{|U|^{g(u)}}{g(u)!}\left(\binom{n}{k-2}p\right)^{g(u)}.$$
\end{lemma}

\begin{proof}
The statement follows if for arbitrary $I\subset W$ and $u\in W\backslash I$ we have
$$\mathbb{P}\left[Z_u=1|\bigwedge_{i\in I} Z_i=1\right]\leq \mathbb{P}\left[\widehat{Z}_u=1|\bigwedge_{i\in I} \widehat{Z}_i=1\right]=\mathbb{P}[\widehat{Z}_u=1]= 1-(1-o(1))\frac{|U|^{g(u)}}{g(u)!}\left(\binom{n}{k-2}p\right)^{g(u)}$$
or equivalently
$$\mathbb{P}\left[Z_u=0|\bigwedge_{i\in I} Z_i=1\right]\geq (1-o(1))\frac{|U|^{g(u)}}{g(u)!}\left(\binom{n}{k-2}p\right)^{g(u)}.$$

Fix $I\subset W$ and $u\in W\backslash I$. Denote by $\Phi_u$ the set of $k$-sets satisfying the following conditions:
\begin{itemize}
\item the $k$-set contains $u$;
\item the $k$-set contains exactly one vertex in $U$;
\item the $k$-set contains no vertex in $V'$;
\item the $k$-set contains no vertex in the set $\{w:g(w)=1\}$.
\end{itemize}
Let $\mathcal{F}_u$ 
be a family of $g(u)$ sized subsets of $\Phi_u$ such that for every $\coll \in \mathcal{F}_u$ we have that any pair of $k$-sets in $\coll$ are vertex disjoint except for $u$. 
For two families of $k$-sets $\coll$ and $\Phi$, we introduce the notation $\mathcal{E}_{\coll,\Phi}$ for the event $\coll\cap \Phi \subseteq E$ and $\mathcal{D}_{\coll,\Phi}$ for the event $(\Phi\setminus \coll)\cap E=\emptyset$.
In particular, for $\coll \in \mathcal{F}_u$ we have $\mathcal{E}_{\coll,\Phi_u}$ is the event $\{\coll \subseteq E\}$ and $\mathcal{D}_{\coll,\Phi_u}$ is the event $\{(\Phi_u\backslash \coll) \cap E=\emptyset\}$.
Thus
\begin{equation}\label{eq:disjsum}
\mathbb{P}\left[Z_u=0|\bigwedge_{i\in I} Z_i=1\right]\geq \sum_{\coll\in \mathcal{F}_u} \mathbb{P}\left[\mathcal{E}_{\coll,\Phi_u},\mathcal{D}_{\coll,\Phi_u}|\bigwedge_{i\in I} Z_i=1\right].
\end{equation}

Let  $\Phi_{u,I}\subseteq \Phi_u$ be the set of $k$-sets which contain at least one vertex in $I$. 
Since the appearance of any $k$-set not containing a vertex in $I$ is independent of the event $\{\bigwedge_{i\in I} Z_i=1\}$, we have for any $\coll \in \mathcal{F}_u$
\begin{equation}\label{eq:lowerZ1}
\mathbb{P}\left[\mathcal{E}_{\coll,\Phi_u},\mathcal{D}_{\coll,\Phi_u},\bigwedge_{i\in I} Z_i=1\right]
= p^{|\coll \backslash \Phi_{u,I}|}(1-p)^{|\Phi_u\backslash \Phi_{u,I}|} \mathbb{P}\left[\mathcal{E}_{\coll,\Phi_{u,I}}, \mathcal{D}_{\coll,\Phi_{u,I}}, \bigwedge_{i\in I} Z_i=1\right].
\end{equation}

For $\coll \in \mathcal{F}_u$ define $U_\coll \subseteq I$ as the set of vertices in $I$ which are contained in a $k$-set in $\coll$. 
Recall that every vertex $v\in U_\coll$ is contained in exactly one $k$-set in $\coll$ and $g(v)\geq 2$. Therefore, conditional on $\mathcal{E}_{\coll,\Phi_u}$, if $v$ is not contained in 
any hyperedge within $E'$ in addition to the hyperedge within $\coll$, then $Z_v=1$. Denote the event that none of the hyperedges in $E'\setminus\mathcal{C}$ contain $v$ by $\mathcal{I}_{\coll,v}$. Then
\begin{equation}\label{eq:lowerZ2}
\mathbb{P}\left[\mathcal{E}_{\coll,\Phi_{u,I}}, \mathcal{D}_{\coll,\Phi_{u,I}}, \bigwedge_{i\in I} Z_i=1\right]
\ge \mathbb{P}\left[\mathcal{E}_{\coll,\Phi_{u,I}}, \mathcal{D}_{\coll,\Phi_{u,I}}, \bigwedge_{v\in U_\coll} \mathcal{I}_{\coll,v}, \bigwedge_{i\in I\backslash U_\coll} Z_i=1\right].
\end{equation}
Out of the four events on the right-hand side of \eqref{eq:lowerZ2} only the first one depends on the hyperedges in $\coll\cap\Phi_{u,I}$ and thus 
\begin{equation}\label{eq:lowerZ3}
\mathbb{P}\left[\mathcal{E}_{\coll,\Phi_{u,I}}, \mathcal{D}_{\coll,\Phi_{u,I}}, \bigwedge_{v\in U_\coll} \mathcal{I}_{\coll,v}, \bigwedge_{i\in I\backslash U_\coll} Z_i=1\right]
=p^{|\coll\cap\Phi_{u,I}|}\mathbb{P}\left[\mathcal{D}_{\coll,\Phi_{u,I}}, \bigwedge_{v\in U_\coll} \mathcal{I}_{\coll,v}, \bigwedge_{i\in I\backslash U_\coll} Z_i=1\right].
\end{equation}
Note that each of the events $\mathcal{D}_{\coll,\Phi_{u,I}}, \bigwedge_{i\in I\backslash U_\coll} Z_i=1$ and $\mathcal{I}_{\coll,v}$ for $v\in U_\coll$ is decreasing, therefore the FKG inequality (Theorem~\ref{thm:FKG}) implies
\begin{equation}\label{eq:lowerZ4}
\mathbb{P}\left[\mathcal{D}_{\coll,\Phi_{u,I}}, \bigwedge_{v\in U_\coll} \mathcal{I}_{\coll,v}, \bigwedge_{i\in I\backslash U_\coll} Z_i=1\right]\geq (1-p)^{|\Phi_{u,I}|} \mathbb{P}\left[\bigwedge_{i\in I\backslash U_\coll} Z_i=1\right]\prod_{v\in U_\coll}\mathbb{P}[\mathcal{I}_{\coll,v}].
\end{equation}
Putting \eqref{eq:lowerZ1}-\eqref{eq:lowerZ4} together we have
\begin{equation}\label{eq:lowerZfinal}
\mathbb{P}\left[\mathcal{E}_{\coll,\Phi_u},\mathcal{D}_{\coll,\Phi_u},\bigwedge_{i\in I} Z_i=1\right]\ge p^{g(u)} (1-p)^{|\Phi_u|} \mathbb{P}\left[\bigwedge_{i\in I\backslash U_\coll} Z_i=1\right]\prod_{v\in U_\coll}\mathbb{P}[\mathcal{I}_{\coll,v}].
\end{equation}

Recall that $|U|=o((n^{k-2}p)^{-1})$ thus
\begin{align}
(1-p)^{|U|\binom{n}{k-2}}&\geq \exp\left(-|U|\binom{n}{k-2}\frac{p}{1-p}\right)\geq \exp(-o(1))=1-o(1).\label{eq:noedge}
\end{align}
As the number of $k$-sets containing $v$ and a vertex in $U$ is at most $|U|\binom{n}{k-2}$, by \eqref{eq:noedge} for every $v\in U_\coll$ we have $\mathbb{P}[\mathcal{I}_{\coll,v}]=1-o(1)$. In addition we have $|U_\coll|\le kg(u)=O(1)$, implying $\prod_{v\in U_\coll}\mathbb{P}[\mathcal{I}_{\coll,v}]=1-o(1)$. Together with $|\Phi_u|\leq |U|\binom{n}{k-2}$ and \eqref{eq:lowerZfinal} this implies
\begin{align*}
\mathbb{P}\left[\mathcal{E}_{\coll,\Phi_u},\mathcal{D}_{\coll,\Phi_u}|\bigwedge_{i\in I} Z_i=1\right]&=\frac{\mathbb{P}[\mathcal{E}_{\coll,\Phi_u},\mathcal{D}_{\coll,\Phi_u},\bigwedge_{i\in I} Z_i=1]}{\mathbb{P}[\bigwedge_{i\in I} Z_i=1]}\\
&\geq (1-o(1))p^{g(u)} (1-p)^{|U|\binom{n}{k-2}} \frac{\mathbb{P}[\bigwedge_{i\in I\backslash U_\coll} Z_i=1]}{\mathbb{P}[\bigwedge_{i\in I\backslash U_\coll} Z_i=1]}\\
&\stackrel{\eqref{eq:noedge}}{\geq} (1-o(1))p^{g(u)}.
\end{align*}

Therefore in order to complete our estimate for the right hand side of \eqref{eq:disjsum} we only need to establish $|\mathcal{F}_u|$. 
Recall that $\mathcal{F}_u$
is the family of $g(u)$ sized subsets of $\Phi_u$ such that for every $\coll \in \mathcal{F}_u$ we have that for any pair of $k$-sets in $\coll$ are vertex disjoint except for $u$.

Since the size of the sets $U,V'$ and $\{w:g(w)=1\}$ is $o(n)$, we have
$$|\Phi_u|\geq |U|\binom{(1-o(1))n}{k-2}.$$
This in turn implies
$$
\binom{|\Phi_u|}{g(u)}\ge(1- o(1))\frac{|U|^{g(u)}}{g(u)!}\binom{n}{k-2}^{g(u)},
$$
as $\binom{(1- o(1))n}{k-2}=(1- o(1))\binom{n}{k-2}$.

Since the number of sets in $\binom{\Phi_u}{g(u)}$ where the two $k$-sets overlap in a vertex other than $u$ is
$$O\left(\left.|U|^{g(u)}\binom{n}{k-2}^{g(u)}\right/|U|\right)=o\left(\binom{|\Phi_u|}{g(u)}\right),$$
we have
$$|\mathcal{F}_u|=(1- o(1))\frac{|U|^{g(u)}}{g(u)!}\binom{n}{k-2}^{g(u)}$$
and the statement follows.

\end{proof}

The previous lemma will be used to establish a lower bound on $\uinfec_i(t)$ when $2\leq i \le r$, and also for $\uinfec(t)$ when $r\ge 3$. However, as mentioned earlier, a different approach is needed for $\uinfec_1(t)$ and $\uinfec(t)$ when $r=2$. Roughly speaking, $\uinfec_1(t)$ increases in every step with the vertices contained in the hyperedges of $\infedge'(t+1)$, and when $r=2$, $\uinfec(t)$ increases with vertices in the closed neighbourhoods of $\uinfec_2(t+1)\setminus \uinfec_2(t)$. By considering how hyperedges intersect, the following lemma ensures that a typical hyperedge contributes $k-1$ vertices to $\uinfec_1(t)$ and an element in $\uinfec_2(t+1)\setminus \uinfec_2(t)$ typically contributes $2k-3$ vertices to $\uinfec(t)$, when $r=2$. 

For a positive integer $\ell\ge 2$ a set of hyperedges $e_1,\ldots,e_\ell$ is called an {\em intersecting $\ell$-tuple with respect to $U$}, if for every $1\le i < \ell$, $e_i$ intersects $e_{i+1}$ in a vertex in $V\setminus U$. If two hyperedges in an intersecting $\ell$-tuple intersect in more than one vertex, which may include a vertex in $U$, we call it a {\em  strongly-intersecting $\ell$-tuple with respect to $U$}. 

\begin{lemma}\label{lem:fewint}
	Consider $H_k(n,p)$ with vertex set $V=[n]$ and hyperedge set $E$. Let $U,V'\subseteq V$ be two disjoint subsets with $|V'|=o(n)$ while  $U$ satisfies both $|U|=\Omega(a_c)$ and $|U|=o((n^{k-2}p)^{-1})$. Denote by $E'\subseteq E$ the set of hyperedges which contain exactly one vertex in $U$ and no vertices in $V'$. 	
 Assume that we are given a set of hyperedges $F$, each containing at least one vertex in $V'$ and are vertex disjoint in $V\setminus (U\cup V')$. Furthermore assume that $|F|=O(|U|n^{k-1}p)$.
	
Then for any $\ell\le r$ with probability $1-o(\exp(- |U|)$ we have that in the hypergraph spanned by $E'\cup F$ the number of intersecting $(\ell+1)$-tuples and the number of strongly-intersecting $\ell$-tuples with respect to $U\cup V'$ is
	$o(|U|^{\ell} n^{\ell(k-2)+1}p^{\ell}).$
\end{lemma}

\begin{proof}
Consider the following process to establish the set of intersecting $(\ell+1)$-tuples and the strongly-intersecting $\ell$-tuples. 
Denote by $\Phi_U$ the set of $k$-sets containing exactly one vertex in $U$, and no vertices in $V'$. Next we create the family $\mathcal{F}_0$ consisting of subsets of $\Phi_U\cup F$, which form an intersecting $(\ell+1)$-tuple or a strongly-intersecting $\ell$-tuple.
In the following we will provide an algorithm which finds every intersecting $(\ell+1)$-tuple and strongly-intersecting $\ell$-tuple reminiscent of the query-process in Section~\ref{sec:query-process}.
Consider the elements $\coll\in\mathcal{F}_0$ one by one, but we do not query every element we consider. An element is queried only if $\coll\cap \Phi_U$ does not contain a hyperedge, which can be found in a previously successfully queried element. In addition the element $\coll\in \mathcal{F}_0$ is successfully queried if it was queried and $\coll \subseteq E' \cup F$. The first stage of the process is complete, once every element in $\mathcal{F}_0$ has been considered. Let $\mathcal{Y}_0$ be the set of hyperedges, which are found in a successfully queried element.

For $i\ge 1$ denote by $\mathcal{Y}_i$ the set of hyperedges in $E'\cup F$, which intersect a hyperedge from $\mathcal{Y}_{i-1}$ in $V\setminus (U\cup V')$. Note that any intersecting $(\ell+1)$-tuple and strongly-intersecting $\ell$-tuple must contain at least one hyperedge in $\mathcal{Y}_0$ and any remaining hyperedges in $\mathcal{Y}_1,\ldots,\mathcal{Y}_i$ for some $i\le \ell$. 
Denote by $\mathcal{F}_i\subseteq \Phi_U\setminus \left(\bigcup_{j=0}^{i-1}\mathcal{Y}_{j}\right)$ the set of $k$-sets, which intersect a hyperedge in $\mathcal{Y}_{i-1}$ in a vertex in $V\setminus (U \cup V')$. 
For $1\le i \le \ell$ repeatedly query every element of $\mathcal{F}_i$ to obtain $\mathcal{Y}_i$.

We start by obtaining an upper bound on $\mathbb{E}[|\mathcal{Y}_0|]$. Lemma~\ref{lem:indepbound} implies that $|\mathcal{Y}_0|$ is stochastically dominated by independent Bernoulli random variables, and denote the sum of these random variables as $Y_0$. Now both $|\Phi_U|p=O(|U|n^{k-1}p)$ and $|F|=O(|U|n^{k-1}p)$, which together with $|U|n^{k-2}p=o(1)$ and $|U|=\omega(1)$ implies:
$$\mathbb{E}[Y_0]=o\left(|U|^{\ell}n^{\ell(k-2)+1}p^{\ell}\right).$$
Together with the concentration result in Theorem~\ref{thm:indconc} we then have
$$\mathbb{P}\left[Y_0=o\left(|U|^{\ell}n^{\ell(k-2)+1}p^{\ell}\right)\right]=1-\exp\left(-\omega \left(|U|^{\ell}n^{\ell(k-2)+1}p^{\ell}\right)\right).$$
Since $\ell\le r$ we have
\begin{align}
|U|^{\ell-1}n^{\ell(k-2)+1}p^{\ell}&=\Omega\left(\left(\frac{1}{n n^{r(k-2)}p^r}\right)^{(\ell-1)/(r-1)}(n^{k-2}p)^{\ell-1}n^{k-1}p\right)\nonumber\\
&=\Omega\left(\left(\frac{1}{ n^{k-1}p}\right)^{(\ell-1)/(r-1)}n^{k-1}p\right)=\Omega(1), \label{eq:Ubound}
\end{align}
and thus
$$\mathbb{P}\left[|\mathcal{Y}_0|=o\left(|U|^{\ell}n^{\ell(k-2)+1}p^{\ell}\right)\right]\ge\mathbb{P}\left[Y_0=o\left(|U|^{\ell}n^{\ell(k-2)+1}p^{\ell}\right)\right]=1-o\left(\exp\left(-|U|\right)\right).$$

Similarly as before by Lemma~\ref{lem:indepbound} we have $|\mathcal{Y}_i\cap E'|$ is stochastically dominated by the binomial random variable $Y_i=\mathrm{Bin}(|U|(k-1)|\mathcal{Y}_{i-1}|n^{k-2},p)$.
Now conditional on $|\mathcal{Y}_{i-1}|=o\left(|U|^{\ell}n^{\ell(k-2)+1}p^{\ell}\right)$ we have $\mathbb{E}[Y_i]=o\left(|U|^{\ell}n^{\ell(k-2)+1}p^{\ell}\right)$ as $|U|n^{k-2}p=o(1)$. Together with Theorem~\ref{thm:indconc} and \eqref{eq:Ubound} we have
\begin{align*}
&\mathbb{P}\left[|\mathcal{Y}_i\cap E'|=o\left(|U|^{\ell}n^{\ell(k-2)+1}p^{\ell}\right)\mid |\mathcal{Y}_{i-1}|=o\left(|U|^{\ell}n^{\ell(k-2)+1}p^{\ell}\right)\right]\\
&\quad \ge \mathbb{P}\left[Y_i=o\left(|U|^{\ell}n^{\ell(k-2)+1}p^{\ell}\right)\mid |\mathcal{Y}_{i-1}|=o\left(|U|^{\ell}n^{\ell(k-2)+1}p^{\ell}\right)\right]=1-o\left(\exp\left(-|U|\right)\right).
\end{align*}

In addition, as the hyperedges in $F$ are vertex disjoint in $V\setminus (U\cup V')$, we also have $|\mathcal{Y}_i\cap F|\le k|\mathcal{Y}_{i-1}|$. 
Therefore with probability $1-o(\exp(-|U|))$ all the $\mathcal{Y}_i$, for $1\le i \le \ell$ have the right size and the result follows.
\end{proof}

\subsection{Bounding trajectory}
Similarly to the subcritical regime we establish a trajectory for the number of infected vertices and in addition for the number of vertices with at least $i$ active neighbours for $1\le i \le r$. 
Recall that $|\infec(0)|=|\uinfec(0)|=(1+\varepsilon)a_c$ and let $0<\delta<1$ satisfy $(1+\varepsilon)(1-\delta)>1$.

\begin{definition}\label{def:trajlower} 
Let $\eta$ and $a^*$ be as defined in \eqref{eq:eta} and \eqref{eq:scaling}. 
Set $\nuinfecappr_0(t):=n$ for all integers $t\ge 0$ and for $1\le i \le r$ let $\nuinfecappr_i(0):=0$. Finally set $\nuinfecappr(-1):=0$ and $\nuinfecappr(0):=|C(0)|$. 
Then for all integers $t\ge 0$ and $1\le i \le r-1$ we define 
\begin{align*}
\nuinfecappr_i(t+1)&:=\sum_{j=0}^{i-1} \nuinfecappr_{j}(t)\frac{(\nuinfecappr(t)-\nuinfecappr(t-1))^{i-j}}{(i-j)!}\left(\binom{n}{k-2}p\right)^{i-j}+\nuinfecappr_i(t);\\
\nuinfecappr_r(t+1)&:=(1-\delta) \sum_{j=0}^{r-1} \nuinfecappr_{j}(t)\frac{(\nuinfecappr(t)-\nuinfecappr(t-1))^{r-j}}{(r-j)!}\left(\binom{n}{k-2}p\right)^{r-j}+\nuinfecappr_r(t);\\
\nuinfecappr(t+1)&:=\eta\nuinfecappr_r(t+1)+\nuinfecappr(0).
\end{align*}
Finally let $\nuinfecrat(t):=\nuinfecappr(t)/a^*$ for all integers $t\ge -1$ and call $\nuinfecrat$ the trajectory associated with mild $r$-bootstrap percolation.
\end{definition} 

We also provide a simpler form for these values.

\begin{claim}\label{clm:traj}
For $0\le i<r$ and $t\ge -1$ we have
$$\nuinfecappr_i(t+1)=\frac{\nuinfecappr(t)^i}{i!}n\left(\binom{n}{k-2}p\right)^i;$$
$$\nuinfecappr_r(t+1)=(1-\delta)\frac{\nuinfecappr(t)^r}{r!}n\left(\binom{n}{k-2}p\right)^r;$$
and
$$\nuinfecappr(t+1)=(1-\delta)\eta\frac{\nuinfecappr(t)^r}{r!}n\left(\binom{n}{k-2}p\right)^r+\nuinfecappr(0).$$
\end{claim}

\begin{proof}
The proof is by induction. Note that the statement holds when $t=-1$ and when $i=0$. Now assume the statement holds for $t$. If $1\le i<r$, then the induction hypothesis for the individual $\nuinfecappr_{j}(t)$ when $j\le i$ leads to 
\begin{align*}
\nuinfecappr_i(t+1)
&=\sum_{j=0}^{i-1} \nuinfecappr_{j}(t)\frac{(\nuinfecappr(t)-\nuinfecappr(t-1))^{i-j}}{(i-j)!}\left(\binom{n}{k-2}p\right)^{i-j}+\nuinfecappr_i(t)\\
&=\sum_{j=0}^{i-1} \frac{\nuinfecappr(t-1)^j}{j!}n\left(\binom{n}{k-2}p\right)^j
\frac{(\nuinfecappr(t)-\nuinfecappr(t-1))^{i-j}}{(i-j)!}\left(\binom{n}{k-2}p\right)^{i-j}+\frac{\nuinfecappr(t-1)^i}{i!}n\left(\binom{n}{k-2}p\right)^i\\
&=n\left(\binom{n}{k-2}p\right)^i \left(\frac{\nuinfecappr(t-1)^i}{i!}+\sum_{j=0}^{i-1} \frac{\nuinfecappr(t-1)^j}{j!}\frac{(\nuinfecappr(t)-\nuinfecappr(t-1))^{i-j}}{(i-j)!}\right).
\end{align*}
For the first statement it only remains to show
$$\nuinfecappr(t-1)^i+i!\sum_{j=0}^{i-1} \frac{\nuinfecappr(t-1)^j}{j!}\frac{(\nuinfecappr(t)-\nuinfecappr(t-1))^{i-j}}{(i-j)!}=\nuinfecappr(t)^i,$$
which follows from the binomial theorem as the left hand side is equivalent to the expansion of $((\nuinfecappr(t)-\nuinfecappr(t-1))+\nuinfecappr(t-1))^i$.
The proof for $\nuinfecappr_r(t)$ is analogous, from which $\nuinfecappr(t)$ follows immediately.
\end{proof}

By an argument analogous to the proof of Claim~\ref{claim:trajectoryRecursion}, that is, replacing $1+\delta$ and $\nlinfecrat$ with $1-\delta$ and $\nuinfecrat$ respectively, we have
\begin{equation}\label{eq:nuratio}
	\nuinfecrat(t+1)=(1-\delta)\frac{\nuinfecrat(t)^r}{r}+\nuinfecrat(0).
\end{equation}

We expect that $\gamma(t)$ does not have a maximum, or equivalently $\nuinfecrat(t+1)-\nuinfecrat(t)$ is bounded from below. 

From \eqref{eq:nuratio}, we obtain the lower bound
\begin{align*}
 \nuinfecrat(t+1)-\nuinfecrat(t)&=(1-\delta)\frac{\nuinfecrat(t)^{r}-\nuinfecrat(t-1)^{r}}{r}\\
 &=(1-\delta)(\nuinfecrat(t)-\nuinfecrat(t-1))\sum_{j=0}^{r-1}\frac{\nuinfecrat(t)^{j}\nuinfecrat(t-1)^{r-j-1}}{r}\\
&\geq (1-\delta)(\nuinfecrat(t)-\nuinfecrat(t-1))\nuinfecrat(t-1)^{r-1}.
\end{align*}
This implies that if there exists a $t$ such that $\nuinfecrat(t-1)^{r-1}>(1-\delta)^{-1}$ or equivalently $\nuinfecappr(t-1)>a^*(1-\delta)^{-1/(r-1)}$, then the sequence $\nuinfecappr(t)$ converges to infinity. In fact, if $\nuinfecrat(t-1)^{r-1}\geq (1+\nu)(1-\delta)^{-1}$ for some fixed $\nu>0$, then $\nuinfecappr(\tau+1)-\nuinfecappr(\tau)$ increases by at least a factor of $1+\nu$ in every step $\tau\ge t$.

In order to prove our result we require not only that $\nuinfecrat(t)$ does not have a maximum, but also that the one step change in each step is sufficiently large to show appropriate tail bounds on a step by step basis.
 
\begin{lemma}\label{lem:mininf}
There exists $\Delta>0$ independent of $n$ such that for every $t\geq 0$ we have $\nuinfecrat(t+1)-\nuinfecrat(t)>\Delta$.
\end{lemma}

\begin{proof}
Using \eqref{eq:nuratio} and that $\nuinfecrat(0)=(1+\varepsilon)a_c/a^*=(1+\varepsilon)(1-r^{-1})$ we have
\begin{align*}
\nuinfecrat(t+1)-\nuinfecrat(t)
&\stackrel{\eqref{eq:nuratio}}{=}(1-\delta)\frac{\nuinfecrat(t)^r}{r}+\nuinfecrat(0)-\nuinfecrat(t)\\
&=(1-\delta)\frac{\nuinfecrat(t)^r}{r}-\nuinfecrat(t)+(1+\varepsilon)\left(1-\frac{1}{r}\right).
\end{align*}
For $x\geq 0$ the function $(1-\delta)x^r/r-x$ has its minimum at $x=(1-\delta)^{-1/(r-1)}$, implying 
\begin{align*}
(1-\delta)\frac{\nuinfecrat(t)^r}{r}-\nuinfecrat(t)+(1+\varepsilon)\left(1-\frac{1}{r}\right)
&\geq \left(\frac{1}{1-\delta}\right)^{1/(r-1)}\left(\frac{1}{r}-1\right)+(1+\varepsilon)\left(1-\frac{1}{r}\right)\\
&=\left(1+\varepsilon-\left(\frac{1}{1-\delta}\right)^{1/(r-1)}\right)\left(1-\frac{1}{r}\right)\\
&> \left(1+\varepsilon-\frac{1}{1-\delta}\right)\left(1-\frac{1}{r}\right).
\end{align*}
Recall that $(1+\varepsilon)(1-\delta)>1$ and thus
$$\Delta:=\left(1+\varepsilon-\frac{1}{1-\delta}\right)\left(1-\frac{1}{r}\right)=\left(\frac{(1+\varepsilon)(1-\delta)-1}{1-\delta}\right)\left(1-\frac{1}{r}\right)>0.$$

\end{proof}
Therefore the value of $\nuinfecrat(t)$ increases by at least $\Delta$ in every step. Thus for any $\nu>0$, in a bounded number of steps, depending only on $\varepsilon,\delta$ and $\nu$, we have $\nuinfecrat(t-1)>((1+\nu)/(1-\delta))^{1/(r-1)}$.

In the last claim of this section we show that if $\nuinfecappr(t)-\nuinfecappr(t-1)$ is small then so is $\nuinfecappr(t)$.

\begin{claim}\label{clm:smallc}
Assume that $\nuinfecappr(t)=\nuinfecappr(t-1)+o((n^{k-2}p)^{-1})$.  Then we have
\begin{equation*}
\nuinfecappr(t)=o((n^{k-2}p)^{-1}).
\end{equation*}
\end{claim}

\begin{proof}
Using $c(t)\le c(t-1)+o((n^{k-2}p)^{-1})$, either $c(t)=o((n^{k-2}p)^{-1})$ or $c(t)=O(c(t-1))$. 
We only have to show that the statement holds in the latter case. By Claim~\ref{clm:traj} we have
$$\nuinfecappr(t)=\Theta\left(\nuinfecappr(t-1)^r n (n^{k-2}p)^r\right),$$
as $\nuinfecappr(0)=O(\nuinfecappr(\tau)^r n ((n^{k-2}p)^r)$ for any $\tau\ge 0$.	
Together with $\nuinfecappr(t)=O(\nuinfecappr(t-1))$ this implies
$$c(t-1)^r=O\left(c(t-1)n^{-1} (n^{k-2}p)^{-r}\right)$$
or equivalently
$$c(t-1)=O\left((n^{k-2}p)^{-1}(n^{k-1}p)^{-1/(r-1)}\right)=o((n^{k-2}p)^{-1}),$$
and we can deduce $c(t)=o((n^{k-2}p)^{-1})$ in this case as well.
\end{proof}

\subsection{Following the trajectory}

In this section we show that as long as the number of vertices added to each of the classes in the previous steps stays above its trajectory this also holds in the next step.
We start by providing the definition for the size of $\act'(t)$, 
\begin{equation}\label{eq:activesize}
	|\act'(t+1)|=
	\left\{
	\begin{array}{ll}
	\min\{\nuinfecappr(t),|\uinfec(t)|\}-|\act(t)| & \mbox{ if } |\uinfec(t)\setminus \act(t)|<(n^{k-2}p\zeta)^{-1}\\
	(n^{k-2}p\zeta)^{-1} & \mbox{ if } (n^{k-2}p\zeta)^{-1}\le |\uinfec(t)\setminus \act(t)|<(n^{k-2}p)^{-1}\zeta^{1/2}\\
	(n^{k-2}p)^{-1} \zeta^{1/2} & \mbox{ if } (n^{k-2}p)^{-1}\zeta^{1/2}\le |\uinfec(t)\setminus \act(t)|,
	\end{array}
	\right.
\end{equation}
where $\zeta=(n^{k-1}p)^{1/(r+1)}$.

Next we calculate the one step expected increase in the size of the sets $\uinfec_i(t)$.

\begin{claim}\label{clm:expect}
Assume that $\nuinfecappr(t)\le |\uinfec(t)|\le \nuinfecappr(t-1)+(n^{k-2}p)^{-1}/\zeta $ and for every $0\le i< r$ we have $|\uinfec_i(t)|\geq (1-o(1))\nuinfecappr_i(t)$. Then 
$$\mathbb{E}[|\uinfec_r(t+1)\setminus \uinfec_{r}(t)|]\ge (1-o(1))(1-\delta)^{-1}(\nuinfecappr_r(t+1)-\nuinfecappr_r(t)).$$
If in addition $|\uinfec_i(t)|\le \nuinfecappr_i(t+1)$ for $2 \le i < r$, then we have
$$\mathbb{E}[|\uinfec_i(t+1)\setminus \uinfec_i(t)|]\ge (1-o(1))(\nuinfecappr_i(t+1)-\nuinfecappr_i(t)).$$
\end{claim}

\begin{proof}
	
Recall that $|\act'(t+1)|\ge\min\{\nuinfecappr(t),|\uinfec(t)|\}-|\act(t)|\ge \nuinfecappr(t)-\nuinfecappr(t-1).$ We are exposing hyperedges which contain exactly one vertex in $\act'(t+1)$ and no vertex in $\uinfec(t)$. By Claim~\ref{clm:smallc} we have $|\uinfec(t)|=o((n^{k-2}p)^{-1})=o(n)$, thus for $2\le i \le r$
\begin{equation}\label{eq:expXhat}
	\mathbb{E}[|\uinfec_i(t+1)\setminus \uinfec_i(t)|]
	\ge\sum_{j=0}^{i-1} (1-o(1))\frac{(\nuinfecappr(t)-\nuinfecappr(t-1))^{i-j}}{(i-j)!}\left(\binom{n}{k-2}p\right)^{i-j}|\uinfec_{j}(t)\setminus (\uinfec_{j+1}(t)\cup \uinfec(t))|.
\end{equation}

As $\nuinfecappr(t)-\nuinfecappr(t-1)=o((n^{k-2}p)^{-1})$, we have  
$$\frac{(\nuinfecappr(t)-\nuinfecappr(t-1))^{i-j}}{(i-j)!}\left(\binom{n}{k-2}p\right)^{i-j}$$
is monotone increasing in $j$. Therefore removing an element of $\uinfec_j(t)\setminus (\uinfec_{j+1}(t)\cup \uinfec(t))$, for some $0\le j \le i-1$ will only decrease the value of \eqref{eq:expXhat}. Then for any sets $C_j'\subseteq \uinfec_j(t)\setminus (C_{j+1}' \cup \uinfec(t))$ for $0\le j \le i-1$ where $C_i'(t)=\uinfec_i(t)$ we have
\begin{align*}
	\mathbb{E}[|\uinfec_i(t+1)\setminus \uinfec_i(t)|]
	&\ge \sum_{j=0}^{i-1} (1-o(1))\frac{(\nuinfecappr(t)-\nuinfecappr(t-1))^{i-j}}{(i-j)!}\left(\binom{n}{k-2}p\right)^{i-j}|\uinfec_j(t)\setminus(\uinfec_{j+1}(t)\cup \uinfec(t))|\\
	&\ge \sum_{j=0}^{i-1} (1-o(1))\frac{(\nuinfecappr(t)-\nuinfecappr(t-1))^{i-j}}{(i-j)!}\left(\binom{n}{k-2}p\right)^{i-j}|\uinfec_j'(t)|.
\end{align*}

Assume that for every $0\le j \le i-1$ one can construct $C_j'$ such that for $0\le j \le i-1$ we have $|C_j'|\ge (1-o(1))\nuinfecappr_j(t)$. Then we have
\begin{align*}
	\mathbb{E}[|\uinfec_i(t+1)\setminus \uinfec_i(t)|]
	&\geq\sum_{j=0}^{i-1} (1-o(1))\frac{(\nuinfecappr(t)-\nuinfecappr(t-1))^{i-j}}{(i-j)!}\left(\binom{n}{k-2}p\right)^{i-j}\nuinfecappr_{j}(t).
\end{align*}
The statement then follows as by definition this last term is $(1-o(1))(\nuinfecappr_i(t+1)-\nuinfecappr_i(t))$, when $2\le i <r$ and $(1-o(1))(1-\delta)^{-1}(\nuinfecappr_r(t+1)-\nuinfecappr_r(t))$ when $i=r$.

In the remainder of the proof we focus on the construction of the aforementioned $C_j'$ and will consider two cases when $t=0$ and $t>0$.

If $t=0$, then $\uinfec_0(0)=V\setminus \uinfec(0)$ and for $1\le j \le i-1$ we have $\uinfec_j(0)=\emptyset$. Setting $C_j'=\uinfec_j(0)$ for $0\le j \le i-1$ implies that $|\uinfec_j'(t)|\ge(1-o(1))\nuinfecappr_j(t)$. 

Now consider $t>0$. 
We will select the sets $C_j'$ consecutively starting with $j=i-1$ and consider the remaining values of $j$ in a decreasing order. 
For every $0\le j \le i-1$ we choose $C_j'$ to contain an arbitrary subset of $\uinfec_j(t)\setminus (C_{j+1}'\cup \uinfec(t))$, with size $(1-o(1))\nuinfecappr_j(t)$. In the following we show that such a subset exists, using induction on $j$.

By Claim~\ref{clm:smallc} we have $|\uinfec(t)|=o((n^{k-2}p)^{-1})$. In addition, for $0\le j \le r-1$ 
$$\nuinfecappr_j(t)\ge \nuinfecappr_{r-1}(1)=\Omega(a_c^{r-1} n (n^{k-2}p)^{r-1})=\Omega((n^{k-2}p)^{-1}),$$
and thus
\begin{equation}\label{eq:fewinfected}
|\uinfec_j(t)\setminus \uinfec(t)|\ge (1-o(1))\nuinfecappr_j(t).
\end{equation}

Note that $|\uinfec_i(t)\setminus \uinfec(t)|\le c_{i}(t+1)$, as either $i<r$ and this follows from our assumption $|\uinfec_i(t)|\le \nuinfecappr_i(t+1)$ or $i=r$, in which case it is a consequence of $\uinfec_r(t)\subseteq \uinfec(t)$. 
Recall that $c_{j+1}(t)$ is monotone increasing in $t$, therefore using \eqref{eq:fewinfected} we have for $0\le j \le i-1$
$$|\uinfec_j(t)\setminus (C_{j+1}'\cup \uinfec(t))|\ge |\uinfec_j(t)\setminus \uinfec(t)|-|C_{j+1}'\setminus \uinfec(t)| \ge (1-o(1))\nuinfecappr_{j}(t)-\nuinfecappr_{j+1}(t+1).$$
By Claim~\ref{clm:traj} and since $c(t)-c(t-1)=o((n^{k-2}p)^{-1})$ for $0\le j \le i-1$ we have 
$$c_{j+1}(t+1)-c_{j+1}(t)=O\left((c(t)-c(t-1))c(t)^{j} n (n^{k-2} p)^{j+1}\right)=o(c(t)^j n (n^{k-2} p)^{j})=o(c_j(t)).$$
Together with Claim~\ref{clm:smallc} a repeated application of Claim~\ref{clm:traj} gives for $0\le j \le i-1$ that
$$c_{j+1}(t)=O(c(t-1)n^{k-2}p c_j(t))=o(c_j(t)) $$
and we conclude that $|\uinfec_j(t)\setminus (C_{j+1}'\cup \uinfec(t))|\ge (1-o(1))\nuinfecappr_{j}(t)$
for every $0\le j \le i-1$. Thus we were able to identify sets $C_j'\subseteq \uinfec_j(t)\setminus (\uinfec(t)\cup C_{j+1}')$ such that $|C_j'|\ge (1-o(1))\nuinfecappr_j(t)$, completing the proof. 
\end{proof}

Define $\Xi(0)=\emptyset$ and $\Xi(t+1)=\Xi(t)\cup\Xi'(t+1)$ where $\Xi'(t+1)\subseteq \infedge'(t+1)$ is the set of hyperedges which intersect another hyperedge of $\infedge'(t+1)$ both in a vertex of $\act'(t)$ and in a vertex of $V\setminus \uinfec(t)$.  
In the following lemma we consider the one step change in our random variables. 

\begin{lemma}\label{lem:uindstep}
	Assume $\nuinfecappr(t)\le |\uinfec(t)|\le \nuinfecappr(t-1)+(n^{k-2}p)^{-1}/\zeta $ and that
	for every $0\le i< r$ we have $|\uinfec_i(t)|\geq (1-o(1))\nuinfecappr_i(t)$. If $|\infedge(t)|\le (1+o(1))\nuinfecappr(t)\binom{n}{k-1}p$ and $\Xi(t)= o\left(\nuinfecappr(t)n^{k-1}p\right)$, 
	then with probability $$1-\exp(-\delta^2(\nuinfecappr(t+1)-\nuinfecappr(t))/16)-o(\exp(-(\nuinfecappr(t)-\nuinfecappr(t-1))))$$
	the following events hold.
	\begin{enumerate}[label=(\roman*)]
		\item[{\em (I)}] \label{cond:1} For every $0\le i<r$ we have $|\uinfec_i(t+1)|\geq (1-o(1))\nuinfecappr_i(t+1)$;
		\item[{\em (II)}] \label{cond:2} $|\uinfec(t+1)|-|\uinfec(t)|\geq \nuinfecappr(t+1)-\nuinfecappr(t)$;
		\item[{\em (III)}] \label{cond:3} $|\infedge(t+1)|\leq (1+o(1))\nuinfecappr(t+1)\binom{n}{k-1}p$;
        \item[{\em (IV)}] \label{cond:4} $\Xi(t+1)= o\left(\nuinfecappr(t+1)n^{k-1}p\right)$.
    \end{enumerate}
\end{lemma}

\begin{proof}
	By definition $|\act'(t+1)|\ge \min\{\nuinfecappr(t),|\uinfec(t)|\}-|\act(t)|\ge \nuinfecappr(t)-\nuinfecappr(t-1).$
	
	We start by proving (I). 
	This holds deterministically when $i=0$ as for every $t\ge 0$ we have $\nuinfecappr_0(t)=n$ and $|\uinfec_0(t)|=|V\setminus \uinfec(0)|=(1-o(1))n$.
	
	Next for $i=1$, let $\mathcal{L}$ denote the set of hyperedges in $\infedge'(t+1)$, which contain only elements in $\uinfec_0(t)\setminus \uinfec_1(t)$. By Claim~\ref{clm:smallc} we have 
 \begin{equation}\label{eq:infedgebound}
 |\uinfec_1(t)|\le k |\infedge(t)|=o(n)
 \end{equation}
 and thus by Theorem~\ref{thm:indconc} with probability $1-o(\exp(-(\nuinfecappr(t)-\nuinfecappr(t-1))))$ we have
	$$|\mathcal{L}|\ge (1-o(1))(\nuinfecappr(t)-\nuinfecappr(t-1))\binom{n}{k-1}p.$$
	Recall that $|\act'(t+1)|\ge \nuinfecappr(t)-\nuinfecappr(t-1)$, which together with Lemma~\ref{lem:mininf} implies $|\act'(t+1)|=\Omega(a^*)=\Omega(a_c)$. In addition Claim~\ref{clm:smallc} implies $|\act'(t+1)|\le \nuinfecappr(t-1)+o((n^{k-2}p)^{-1}) =o((n^{k-2}p)^{-1})$. 
	Together with \eqref{eq:infedgebound} we have verified that the conditions for Lemma~\ref{lem:fewint} are met when $U=\act'(t+1)$, $V'=\uinfec_1(t)\setminus \act'(t+1)$ and $F=\emptyset$. So with probability $1-o(\exp(-(\nuinfecappr(t)-\nuinfecappr(t-1))))$ almost every pair of hyperedges in $\mathcal{L}$ is vertex disjoint (except possibly for a vertex in $\act'(t+1)$), and we not only have 
 $$|\uinfec_1(t+1)\setminus\uinfec_1(t)|\geq (1-o(1))(\nuinfecappr_1(t+1)-\nuinfecappr_1(t)),$$
 but also  (IV).

		Now fix $2\le i \le r$. 
	For $u\in V\setminus (\uinfec_i(t)\cup \uinfec(t))$ let $X_{u,i}$ be the indicator random variable for the event $u\in \uinfec_i(t+1)$. Let $X_{i}=\sum_{u\in V\setminus(\uinfec_i(t)\cup\uinfec(t))} X_{u,i}$. 
	Assign for every $u\in V\setminus (\uinfec_{i}(t)\cup \uinfec(t))$ an independent Bernoulli random variable $\widehat{X}_{u,i}$ such that  
	$$\mathbb{P}[\widehat{X}_{u,i}=1]=(1-o(1))\frac{(\nuinfecappr(t)-\nuinfecappr(t-1))^{i-j}}{(i-j)!}\left(\binom{n}{k-2}p\right)^{i-j},$$
	and denote by $\widehat{X}_i=\sum_{u\in V\setminus (\uinfec_i(t)\cup \uinfec(t))}\widehat{X}_{u,i}$. Then $\mathbb{E}[\widehat{X}_{u,i}]=(1-o(1))\mathbb{E}[X_{u,i}]$ and thus also $\mathbb{E}[\widehat{X}_i]=(1-o(1))\mathbb{E}[X_i]$.
	
	Define $Z_{u,i}=1-X_{u,i}$ and $\widehat{Z}_{u,i}=1-\widehat{X}_{u,i}$. Similarly as before we will also need the sums $Z_i=\sum_{u\in V\setminus (\uinfec_i(t)\cup \uinfec(t))} Z_{u,i}$ and $\widehat{Z}_i=\sum_{u\in V\setminus (\uinfec_i(t)\cup \uinfec(t))} \widehat{Z}_{u,i}$.
	
	Next we apply Lemma~\ref{lem:negcorrappl}. Recall that $\act'(t+1)$ satisfies both $|\act'(t+1)|=\Omega(a_c)$ and $|\act'(t+1)|=o((n^{k-2}p)^{-1})$. Also by \eqref{eq:infedgebound} and $|\uinfec(0)|=o(n)$ we have $|\uinfec(t)|=o(n)$. Set $U=\act'(t+1)$, $V'=\uinfec(t)\setminus \act'(t+1)$, $W=V\setminus (\uinfec_i(t)\cup U\cup V')$ and for $w\in W\cap (\uinfec_j(t)\setminus \uinfec_{j+1}(t))$ let $g(w)=i-j$. By \eqref{eq:infedgebound} we also have $|\uinfec_{i-1}(t)\setminus \uinfec_i(t)|\leq |\uinfec_1(t)|=o(n)$, and thus $|\{w\in W:g(w)=1\}|=o(n)$.
    Then
	Lemma~\ref{lem:negcorrappl} implies for every $I\subseteq W$ that
	\begin{equation*}
	\mathbb{E}\left[\prod_{u\in I} Z_{u,i}\right]\leq \mathbb{E}\left[\prod_{u\in I} \widehat{Z}_{u,i}\right].
	\end{equation*}
	Note that $\mathrm{Var}(\widehat{Z}_i)=\mathrm{Var}(\widehat{X}_i)\leq \mathbb{E}[\widehat{X}_i]$ thus Lemma~\ref{lem:negcorr} implies for any $\lambda>0$ that 
	\begin{align}
	\mathbb{P}\left[X_i\leq (1-\lambda) \mathbb{E}[ \widehat{X}_i]\right]&=\mathbb{P}\left[Z_i\geq \mathbb{E}[ \widehat{Z}_i]+\lambda \mathbb{E}[\widehat{X}_i]\right]\nonumber\\
	&\leq \exp\left(-(1-o(1))\lambda^2\frac{\mathbb{E}[\widehat{X}_i]^2}{2(\mathbb{E}[\widehat{X}_i]+\mathbb{E}[\widehat{X}_i]/3)}\right)\nonumber\\
	&\leq\exp\left(-\lambda^2\frac{\mathbb{E}[\widehat{X}_i]}{3}\right).\label{eq:expsmallprob}
	\end{align}
	
	In order to complete the proof of (I) 
	we consider the case $1<i<r$. 	
	Let $\lambda$ satisfy 
	$$\left(\frac{\nuinfecappr(t+1)-\nuinfecappr(t)}{\nuinfecappr_i(t+1)-\nuinfecappr_i(t)} \right)^{1/2}\ll \lambda \ll 1,$$
	and note that such a $\lambda$ exists.
    Note that we may assume that $|\uinfec_i(t)|\leq \nuinfecappr_i(t+1)$, as otherwise (I) follows. 
	Then by \eqref{eq:expsmallprob}, Claim~\ref{clm:expect} and $\mathbb{E}[\widehat{X}_i]=(1-o(1))\mathbb{E}[X_i]$ we have 
	\begin{align*}
	\mathbb{P}[X_i\leq (1-o(1))(1-\lambda)(\nuinfecappr_i(t+1)-\nuinfecappr_i(t))]&\leq \exp\left(-(1-o(1))\lambda^2\frac{(\nuinfecappr_i(t+1)-\nuinfecappr_i(t))}{3}\right)\\
	&=o\left(\exp\left(-\delta^2(\nuinfecappr(t+1)-\nuinfecappr(t))/16\right)\right)
	\end{align*}
	due to our choice of $\lambda$. Completing the proof of (I). 
	
	Now if $i=r$ we select $\lambda=\delta/2$. Therefore by \eqref{eq:expsmallprob}, Claim~\ref{clm:expect} and $$\mathbb{E}[\widehat{X}_i]=(1-o(1))\mathbb{E}[X_i]=(1-o(1))\mathbb{E}[|\uinfec_i(t+1)\setminus \uinfec_i(t)|]$$
	 we have
	$$\mathbb{P}\left[X_r\leq \left(1+\frac{\delta}{2}\right)(\nuinfecappr_r(t+1)-\nuinfecappr_r(t))\right]\leq \exp\left(-(1+o(1))\delta^2 \frac{\nuinfecappr(t+1)-\nuinfecappr(t)}{12}\right).$$

	When $r\ge 3$ one can derive (II) 
	as $\uinfec(t+1)\setminus \uinfec(t)=\uinfec_r(t+1)\setminus \uinfec_r(t)$, for every $t\ge 0$. 
    
    On the other hand, if $r=2$, we have $\uinfec(t+1)\setminus\uinfec(t)$ consists of the vertices in $\uinfec_2(t+1)\setminus \uinfec_2(t)$ and their neighbours in the hypergraph spanned by $\infedge(t+1)$. Our goal is to show that with probability $1-o(\exp(\nuinfecappr(t)-\nuinfecappr(t-1)))$ almost every vertex in $\uinfec_2(t+1)\setminus \uinfec_2(t)$ is responsible for $2k-3$ infections. For a vertex to be involved in less than $2k-3$ vertices becoming infected, it has to be either part of an intersecting 3-tuple, a strongly intersecting 2-tuple or within a hyperedge of $\Xi(t)$. 
    Using Theorem~\ref{thm:indconc} and 
    $$n^{k-1}p \nuinfecappr(t)(\nuinfecappr(t)-\nuinfecappr(t-1))n^{k-2}p=\Theta(\nuinfecappr(t)-\nuinfecappr(t-1))=\Theta(\nuinfecappr_2(t)-\nuinfecappr_2(t-1)),$$
     we have with probability $1-o(\exp(-(\nuinfecappr(t)-\nuinfecappr(t-1))))$ that the number of hyperedges intersecting $\Xi(t)$ is $o(\nuinfecappr_2(t)-\nuinfecappr_2(t-1))$.
    Note that in $\infedge(t)\setminus \Xi(t)$ either every vertex is infected or exactly one vertex is infected, and thus every hyperedge in $\infedge(t)\setminus \Xi(t)$, which still contains uninfected vertices, must be vertex disjoint, except for possibly a vertex in $\act(t)$. By Lemma~\ref{lem:fewint} with probability $1-o(\exp(-(\nuinfecappr(t)-\nuinfecappr(t-1)))$ both the number of intersecting 3-tuples and the number of strongly intersecting 2-tuples, which rely on the hyperedges in $\infedge'(t+1)$ and $\infedge(t)\setminus \Xi(t)$ is $o(\nuinfecappr_2(t+1)-\nuinfecappr_2(t))$. Therefore with probability $1-o(\exp(-(\nuinfecappr(t)-\nuinfecappr(t-1)))$ almost every vertex in $\uinfec_2(t+1)\setminus (\uinfec_2(t)\cup \uinfec(t))$ results in $2k-3$ infections, completing the proof of (II). 

	Finally, by Theorem~\ref{thm:indconc} we have with probability $1-o(\exp(-(\nuinfecappr(t)-\nuinfecappr(t-1)))$ we have
	$$\infedge'(t+1)\le (1+o(1))(\nuinfecappr(t)-\nuinfecappr(t-1))\binom{n}{k-1}p,$$
	proving (III). 
\end{proof}

\subsection{The last two steps}
In the following two lemmas we consider what happens in mild $r$-bootstrap percolation after $|\uinfec(t)\setminus \act(t)|$  becomes sufficiently large, and show that almost every vertex becomes infected in the next two steps.

\begin{lemma}\label{lem:penultimatestep}
	Consider $H_k(n,p)$ with vertex set $V=[n]$ and hyperedge set $E$. Let $U,V'\subseteq V$ be two disjoint subsets with $|V'|=o(n)$ while  $|U|=(n^{k-2}p)^{-1}/\zeta$. Denote by $E'\subseteq E$ the set of hyperedges which contain exactly one vertex in $U$ and no vertices in $V'$. 	
	Denote by $E'\subseteq E$ the set of hyperedges which contain exactly one vertex in $U$ and no vertex in $V'$. Then with probability $1-o(\exp(-a^*))$ the number of vertices in $V\setminus (U\cup V')$ with at least $r$ neighbours in $U$, in the subgraph spanned by $E'$, is at least $(n^{k-2}p)^{-1}\zeta^{1/2}$ and $|E'|=o(n)$.
\end{lemma}

\begin{proof}
For $u\in V \setminus (U \cup V')$ let $X_u$ be the indicator random variable that $u$ has at least $r$ neighbours in $U$. Let $\widehat{X}_u$ be independent Bernoulli random variables such that
$$\mathbb{P}[\widehat{X}_{u}=1]=(1-o(1))\frac{|U|^{r}}{r!}\left(\binom{n}{k-2}p\right)^{r}.$$
In addition, define $Z_u=1-X_u$ and $\widehat{Z}_u=1-\widehat{X}_u$, and denote by $X,\widehat{X},Z,\widehat{Z}$ the sum over $V\setminus (U\cup V')$ of the $X_u,\widehat{X}_u,Z_u,\widehat{Z}_u$, respectively. 
Note that $|U|=o((n^{k-2}p)^{-1})$ and 
$$ |U|=\frac{1}{n^{k-2}p\zeta}=\frac{1}{n^{k-2}p (n^{k-1}p)^{1/(r+1)}}=\omega\left(\left(\frac{1}{n\left(n^{k-2}p\right)^r}\right)^{1/(r-1)}\right)=\omega(a^*).$$
Then by Lemma~\ref{lem:negcorrappl} (setting $g(w)=r$ for every $w\in V\setminus (U\cup V')$) and using Lemma~\ref{lem:negcorr} with $\mathrm{Var}(\widehat{Z})=\mathrm{Var}(\widehat{X})\leq \mathbb{E}[\widehat{X}]$ implies
	\begin{align*}
	\mathbb{P}\left[X\leq \mathbb{E}[ \widehat{X}]/2\right]&=\mathbb{P}\left[Z\geq \mathbb{E}[ \widehat{Z}]+\mathbb{E}[\widehat{X}]/2\right]\leq\exp\left(-\frac{\mathbb{E}[\widehat{X}]}{12}\right).
\end{align*}

The first part of the statement follows as
\begin{align*}
\mathbb{E}[\widehat{X}]=(1-o(1))n\binom{|U|}{r}\left(\binom{n}{k-2}p\right)^r&=\Omega\left( n\left(\frac{1}{n^{k-2}p \zeta}\right)^r (n^{k-2}p)^r \right)=\Omega\left(\frac{1}{(n^{k-2}p)}\zeta \right)=\omega(a^*).
\end{align*}

As for the bound on $|E'|$, this follows from Theorem~\ref{thm:indconc} and $n/\zeta\ge (n^{k-2}p\zeta)^{-1}=\omega(a^*)$.

\end{proof}

\begin{lemma}\label{lem:finalstep}
		Consider $H_k(n,p)$ with vertex set $V=[n]$ and hyperedge set $E$. Let $U,V'\subseteq V$ be two disjoint subsets with $|V'|=o(n)$ while $|U|=(n^{k-2}p)^{-1}\zeta^{1/2}$. 
		Denote by $E'\subseteq E$ the set of hyperedges which contain exactly one vertex in $U$ and no vertex in $V'$. Then with probability $1-o(\exp(-a^*))$ there are $o(n)$ vertices in $V\setminus (U \cup V')$ which have less than $r$ neighbours in $U$, within the subgraph spanned by $E'$.
\end{lemma}

\begin{proof}
Consider an arbitrary ordering on vertices in $V\setminus (U\cup V')$, and denote the $i$-th element by $v_i$. Let $Z_i$ be the indicator random variable that $v_i$ has less than $r$ neighbours in $U$. 
Define the martingale $M_{0}=0$ and for $i\ge 1$
$$M_{i}=M_{i-1}+Z_{i}-\mathbb{E}[Z_{i}\mid Z_1,\ldots,Z_{i-1}].$$
Let $(z_1,...,z_{i-1})\in \{0,1\}^n$ such that $\sum_{j=1}^{i-1}z_j=o(n)$ and for $\ell=0,1$ denote by $W_\ell=\{j<i:z_j=\ell\}$. Let $\Phi$ be the set of $k$-sets, which contain a vertex in $U$, a vertex $v_j$ with $j< i$ such that $z_j=1$, and no vertex in $V'$, that is, $\Phi$ is the set of $k$-sets which are exposed and contain a vertex in $W_1$. Finally for $\Psi\subseteq \Phi$ let $\mathcal{E}_\Psi$ be the event that $\infedge'(t+1)\cap\Phi=\Psi$.
Then, for any $\Psi$ such that $\mathcal{E}_{\Psi}\subseteq \bigwedge_{j\in W_1}Z_j=z_j$, by applying the FKG inequality (Thyeorem~\ref{thm:FKG}) in the conditional probability space $\mathcal{E}_{\Psi}$ we have
$$
\mathbb{E}\left[Z_{i}\mid \bigwedge_{j=1}^{i-1}Z_j=z_j, \mathcal{E}_{\Psi}\right]=\mathbb{E}\left[Z_{i}\mid \bigwedge_{j\in W_0}Z_j=z_j, \mathcal{E}_{\Psi}\right]\le \mathbb{E}[Z_{i}\mid \mathcal{E}_{\Psi}]\le \mathbb{E}[Z_{i}\mid \mathcal{E}_{\emptyset}],
$$
where in the last step we use that removing hyperedges from the hypergraph increases the probability of the event $\{Z_i=1\}$. As this holds for every $\Psi\subseteq\Phi$ satisfying $\mathcal{E}_{\Psi}\subseteq \bigwedge_{j\in W_1}Z_j=z_j$ and exposing the hyperedges within $\Phi$ determines the values of $\{Z_j\}_{j\in W_1}$ we also have
\begin{equation}\label{eq:Zexp1}
\mathbb{E}\left[Z_{i}\mid \bigwedge_{j=1}^{i-1}Z_j=z_j\right]\le \mathbb{E}[Z_{i}\mid \mathcal{E}_{\emptyset}].    
\end{equation}

For $u\in U$ let $\mathcal{N}_{u}$ be the event that $u$ is a neighbour of $v_i$. 
Since the event $\mathcal{E}_{\emptyset}$ excludes every hyperedge, which contains one of tbe $o(n)$ vertices associated with $W_1$ we have
$$\mathbb{P}[\mathcal{N}_{u}\mid \mathcal{E}_{\emptyset}]\ge 1-(1-p)^{(1-o(1))\binom{n}{k-2}}\ge 1-\exp\left((1-o(1))\binom{n}{k-2}p\right)=(1-o(1))\binom{n}{k-2}p.$$
Note that in the conditional space $\mathcal{E}_{\emptyset}$ the events $\mathcal{N}_u$ are mutually independent, as each hyperedge contains one vertex in $u$. In addition a vertex can have at most $r-1$ neighbours in $U$ when $Z_{i}=1$. Since the expected number of vertices $u\in U$ for which $\mathcal{N}_u$ holds is $\Omega(\zeta^{1/2})$ by Theorem~\ref{thm:indconc} we have
\begin{equation}\label{eq:Zexp2}
\mathbb{P}[Z_i=1\mid \mathcal{E}_{\emptyset}]\le e^{-\zeta^{1/4}}.
\end{equation}

Next we introduce a stopping time for the martingale
$$M_i'=\left\{\begin{array}{ll}
M_i &\mbox{if } \sum_{j=1}^{i-1}Z_j\le (na_c)^{1/2}+ne^{-\zeta^{1/4}} \quad \mbox{and} \quad  M_{i-1}'\le (n a_c)^{1/2}\\
M_{i-1}' &\mbox{if } \sum_{j=1}^{i-1}Z_j> (n a_c)^{1/2}+ne^{-\zeta^{1/4}} \quad \mbox{or } \quad M_{i-1}'> (n a_c)^{1/2}.
\end{array}
\right.$$
Azuma's inequality (Theorem~\ref{thm:azuma}) implies that
$$\mathbb{P}[M'_{|V'\setminus U|}> (na_c)^{1/2}]\le \exp\left(-a_c/2\right)$$
and thus we have with probability $1- e^{-a_c/2}$ that $M_i'\le (n a_c)^{1/2}$ for every $1\le i\le |V\setminus (U\cup V')|$.

Now assume $M_i'\le (n a_c)^{1/2}$ for every $1\le i\le |V\setminus (U\cup V')|$. We will show inductively that $M_i=M_{i}'$ for every $0\le i \le |V\setminus (U\cup V')|$.
Clearly this holds if $i=0$. 
Using the induction hypothesis for $i-1$ together with \eqref{eq:Zexp1} and \eqref{eq:Zexp2} implies
$$\sum_{j=1}^{i-1}Z_j \le  (n a_c)^{1/2}+ne^{-\zeta^{1/4}},$$
and thus $M_i=M_i'$. Using $M_{|V\setminus (U\cup V')|}=M_{|V\setminus (U\cup V')|}'\le (na_c)^{1/2}$ an analogous argument gives
$$\sum_{j=1}^{|V\setminus (U\cup V')|}Z_j \le  (n a_c)^{1/2}+ne^{-\zeta^{1/4}}=o(n),$$
completing the proof.
\end{proof}

\subsection{Proof of Theorem~\ref{thm:hypergraph}\,\ref{case:super} }

\begin{proof}[Proof of Theorem~\ref{thm:hypergraph}~\ref{case:super}]
	Recall that $\zeta=(n^{k-1}p)^{1/(r+1)}$ and note that $(n^{k-2}p)^{-1}/\zeta=\omega(a^*)$. Denote by $t_0$ the first step $t$ such that $|\uinfec(t)\setminus \act(t)|\geq (n^{k-2}p)^{-1}/\zeta$. 
	We will apply Lemma~\ref{lem:uindstep} repeatedly until step $t_0$. 
	Note that if we assume that the statement of Lemma~\ref{lem:uindstep} holds for every step $\tau\le t$ for some $t<t_0-1$ then, as this implies $|\act(t)|=\nuinfecappr(t)$, the conditions of Lemma~\ref{lem:uindstep} are satisfied in the following step as well. By Lemma~\ref{lem:mininf} after a bounded number of steps we have $(\nuinfecappr(t+1)-\nuinfecappr(t))/(\nuinfecappr(t)-\nuinfecappr(t-1))>1+\nu$ for some fixed $\nu$. Therefore there exists a step $t_1$, such that $\nuinfecappr(t_1+1)-\nuinfecappr(t_1)\geq (n^{k-2}p)^{-1}/\zeta$. Under the assumption that  the statement of Lemma~\ref{lem:uindstep} holds for every step $\tau\le t_0$ we also have $t_0\le t_1$.
	
	Next we establish the probability that the statement of Lemma~\ref{lem:uindstep} holds in every step up to $t_0$.
	By the union bound this probability is at least
	$$1-\sum_{t=0}^{\infty} \exp(-\delta^2(\nuinfecappr(t+1)-\nuinfecappr(t))/16)-o(\exp(-(\nuinfecappr(t)-\nuinfecappr(t-1))).$$
	Again using the fact that after a bounded number of steps we have that $(\nuinfecappr(t+1)-\nuinfecappr(t))/(\nuinfecappr(t)-\nuinfecappr(t-1))>1+\nu$ for some fixed $\nu$ and since Lemma~\ref{lem:mininf} implies $\nuinfecappr(t+1)-\nuinfecappr(t)=\Omega(a^*)$ we conclude
	$$1-\sum_{t=0}^{\infty} \exp(-\delta^2(\nuinfecappr(t+1)-\nuinfecappr(t))/16)-o(\exp(-(\nuinfecappr(t)-\nuinfecappr(t-1)))=1-\exp(-\Omega(a^*)).$$

    By Claim~\ref{clm:smallc} we have $|\act(t_0)|=\nuinfecappr(t_0)=o((n^{k-2}p)^{-1})$, which together with Lemma~\ref{lem:uindstep}\,(III) implies $|\uinfec(t_0)|\le k|\infedge(t_0)|+|\uinfec(0)|=o(n)$. 
    
    Next we apply Lemma~\ref{lem:penultimatestep} where $U$ is an arbitrary $(n^{k-2}p)^{-1}/\zeta$ sized subset of $\uinfec(t_0)\setminus \act(t_0)$ and $V'=\uinfec(t_0)$. Thus with probability $1-\exp(-a^*)$ at least $(n^{k-2}p)^{-1}\zeta^{1/2}$, but at most $o(n)$ additional vertices become infected. 
    
    As the number of infected vertices at the end of step $t_0$ was $o(n)$ and in step $t_0+1$ this was increased by an additional $o(n)$ infected vertices the result follows from Lemma~\ref{lem:finalstep}.
\end{proof}

{\bf Acknowledgements.} The first and second authors are supported in part by the Austrian Science Fund (FWF) [10.55776/\{W1230, DOC183, I6502\}]. For the purpose of open access, the author has applied a CC BY public copyright licence to any Author Accepted Manuscript version arising from this submission.

\bibliographystyle{plain}
\bibliography{referencehyp}

\begin{thebibliography}{10}

\bibitem{MR3783205}
M.~A. Abdullah and N.~Fountoulakis.
\newblock A phase transition in the evolution of bootstrap percolation
  processes on preferential attachment graphs.
\newblock {\em Random Structures Algorithms}, 52(3):379--418, 2018.

\bibitem{MR968311}
M.~Aizenman and J.~Lebowitz.
\newblock Metastability effects in bootstrap percolation.
\newblock {\em J. Phys. A}, 21(19):3801--3813, 1988.

\bibitem{MR2595485}
H.~Amini.
\newblock Bootstrap percolation and diffusion in random graphs with given
  vertex degrees.
\newblock {\em Electron. J. Combin.}, 17(1):Research Paper 25, 20, 2010.

\bibitem{MR2728841}
H.~Amini.
\newblock Bootstrap percolation in living neural networks.
\newblock {\em J. Stat. Phys.}, 141(3):459--475, 2010.

\bibitem{bootpower}
H.~Amini and N.~Fountoulakis.
\newblock Bootstrap percolation in power-law random graphs.
\newblock {\em Journal of Statistical Physics}, 155(1):72--92, 2014.

\bibitem{arXiv:1402.2815}
H.~Amini, N.~Fountoulakis, and K.~Panagiotou.
\newblock Bootstrap percolation in inhomogeneous random graphs.
\newblock {\em Advances in Applied Probability}, pages 1--49, 2023.

\bibitem{MR3784495}
O.~Angel and B.~Kolesnik.
\newblock Sharp thresholds for contagious sets in random graphs.
\newblock {\em Ann. Appl. Probab.}, 28(2):1052--1098, 2018.

\bibitem{MR4328426}
O.~Angel and B.~Kolesnik.
\newblock {Large deviations for subcritical bootstrap percolation on the
  Erd\H{o}s-R\'{e}nyi graph}.
\newblock {\em J. Stat. Phys.}, 185(2):Paper No. 8, 16, 2021.

\bibitem{MR2214907}
J.~Balogh and B.~Bollob{\'a}s.
\newblock Bootstrap percolation on the hypercube.
\newblock {\em Probab. Theory Related Fields}, 134(4):624--648, 2006.

\bibitem{MR2888224}
J.~Balogh, B.~Bollob{\'a}s, H.~Duminil-Copin, and R.~Morris.
\newblock The sharp threshold for bootstrap percolation in all dimensions.
\newblock {\em Trans. Amer. Math. Soc.}, 364(5):2667--2701, 2012.

\bibitem{MR2546747}
J.~Balogh, B.~Bollob{\'a}s, and R.~Morris.
\newblock Bootstrap percolation in three dimensions.
\newblock {\em Ann. Probab.}, 37(4):1329--1380, 2009.

\bibitem{BALOGH20121328}
J.~Balogh, B.~Bollobás, R.~Morris, and O.~Riordan.
\newblock Linear algebra and bootstrap percolation.
\newblock {\em Journal of Combinatorial Theory, Series A}, 119(6):1328--1335,
  2012.

\bibitem{MR2248323}
J.~Balogh, Y.~Peres, and G.~Pete.
\newblock Bootstrap percolation on infinite trees and non-amenable groups.
\newblock {\em Combin. Probab. Comput.}, 15(5):715--730, 2006.

\bibitem{MR2283230}
J.~Balogh and B.~Pittel.
\newblock Bootstrap percolation on the random regular graph.
\newblock {\em Random Structures Algorithms}, 30(1-2):257--286, 2007.

\bibitem{MR0244077}
B.~Bollob\'{a}s.
\newblock Weakly {$k$}-saturated graphs.
\newblock In {\em Beitr\"{a}ge zur {G}raphentheorie ({K}olloquium, {M}anebach,
  1967)}, pages 25--31. B. G. Teubner Verlagsgesellschaft, Leipzig, 1968.

\bibitem{MR3426518}
E.~Candellero and N.~Fountoulakis.
\newblock Bootstrap percolation and the geometry of complex networks.
\newblock {\em Stochastic Process. Appl.}, 126(1):234--264, 2016.

\bibitem{MR1921442}
R.~Cerf and F.~Manzo.
\newblock The threshold regime of finite volume bootstrap percolation.
\newblock {\em Stochastic Process. Appl.}, 101(1):69--82, 2002.

\bibitem{bootstrapintr}
J.~Chalupa, P.~L. Leath, and G.~R. Reich.
\newblock Bootstrap percolation on a {Bethe} lattice.
\newblock {\em Journal of Physics C: Solid State Physics}, 12:L31--L35, 1979.

\bibitem{MR3451155}
A.~Coja-Oghlan, U.~Feige, M.~Krivelevich, and D.~Reichman.
\newblock Contagious sets in expanders.
\newblock In {\em Proceedings of the {T}wenty-{S}ixth {A}nnual {ACM}-{SIAM}
  {S}ymposium on {D}iscrete {A}lgorithms}, pages 1953--1987. SIAM,
  Philadelphia, PA, 2015.

\bibitem{CZ2022}
O.~Cooley and J.~Zalla.
\newblock High-order bootstrap percolation in hypergraphs.
\newblock {\em Preprint}, 2022.
\newblock arXiv:2201.09718.

\bibitem{MR3992287}
N.~Detering, T.~Meyer-Brandis, and K.~Panagiotou.
\newblock Bootstrap percolation in directed inhomogeneous random graphs.
\newblock {\em Electron. J. Combin.}, 26(3):Paper No. 3.12, 43, 2019.

\bibitem{MR0253433}
M.~Dwass.
\newblock The total progeny in a branching process and a related random walk.
\newblock {\em J. Appl. Probability}, 6:682--686, 1969.

\bibitem{inhbootstrap}
H.~Einarsson, J.~Lengler, F.~Mousset, K.~Panagiotou, and A.~Steger.
\newblock Bootstrap percolation with inhibition.
\newblock {\em Random Structures \& Algorithms}, 55(4):881--925, 2019.

\bibitem{falgas-ravry_sarkar_2023}
V.~Falgas-Ravry and A.~Sarkar.
\newblock Bootstrap percolation in random geometric graphs.
\newblock {\em Advances in Applied Probability}, page 1–47, 2023.

\bibitem{MR3719944}
U.~Feige, M.~Krivelevich, and D.~Reichman.
\newblock Contagious sets in random graphs.
\newblock {\em Ann. Appl. Probab.}, 27(5):2675--2697, 2017.

\bibitem{MR2430783}
L.~Fontes and R.~Schonmann.
\newblock Bootstrap percolation on homogeneous trees has 2 phase transitions.
\newblock {\em J. Stat. Phys.}, 132(5):839--861, 2008.

\bibitem{Fontes02stretchedexponential}
L.~Fontes, R.~Schonmann, and V.~Sidoravicius.
\newblock Stretched exponential fixation in stochastic {Ising} models at zero
  temperature.
\newblock {\em Comm. Math. Phys}, 228:495--518, 2002.

\bibitem{MR3784494}
N.~Fountoulakis, M.~Kang, C.~Koch, and T.~Makai.
\newblock A phase transition regarding the evolution of bootstrap processes in
  inhomogeneous random graphs.
\newblock {\em Ann. Appl. Probab.}, 28(2):990--1051, 2018.

\bibitem{MR3668849}
K.~Gunderson, S.~Koch, and M.~Przykucki.
\newblock The time of graph bootstrap percolation.
\newblock {\em Random Structures Algorithms}, 51(1):143--168, 2017.

\bibitem{MR1961342}
A.~Holroyd.
\newblock Sharp metastability threshold for two-dimensional bootstrap
  percolation.
\newblock {\em Probab. Theory Related Fields}, 125(2):195--224, 2003.

\bibitem{JLR}
S.~Janson, T.~{\L}uczak, and A.~Ruci\'nski.
\newblock {\em Random graphs}.
\newblock Wiley-Interscience Series in Discrete Mathematics and Optimization.
  Wiley-Interscience, New York, 2000.

\bibitem{MR3025687}
S.~Janson, T.~{\L}uczak, T.~Turova, and T.~Vallier.
\newblock Bootstrap percolation on the random graph {$G_{n,p}$}.
\newblock {\em Ann. Appl. Probab.}, 22(5):1989--2047, 2012.

\bibitem{KANG2015595}
M.~Kang, C.~Koch, and T.~Makai.
\newblock Bootstrap percolation in random $k$-uniform hypergraphs.
\newblock {\em Electronic Notes in Discrete Mathematics}, 49:595--601, 2015.
\newblock The Eight European Conference on Combinatorics, Graph Theory and
  Applications, EuroComb 2015.

\bibitem{MR3817532}
M.~Kang and T.~Makai.
\newblock Bootstrap percolation on {$G(n,p)$} revisited.
\newblock In {\em Proceedings of the 27th {I}nternational {C}onference on
  {P}robabilistic, {C}ombinatorial and {A}symptotic {M}ethods for the
  {A}nalysis of {A}lgorithms---{A}of{A}'16}, page~12. Jagiellonian Univ., Dep.
  Theor. Comput. Sci., Krak\'{o}w, 2016.

\bibitem{MR4206553}
C.~Koch and J.~Lengler.
\newblock Bootstrap percolation on geometric inhomogeneous random graphs.
\newblock {\em Internet Math.}, page~36, 2021.

\bibitem{MR4372097}
B.~Kolesnik.
\newblock {The sharp $K_{4}$-percolation threshold on the
  {E}rd\H{o}s-{R}\'{e}nyi random graph}.
\newblock {\em Electron. J. Probab.}, 27:Paper No. 13, 23, 2022.

\bibitem{MR1678578}
C.~McDiarmid.
\newblock Concentration.
\newblock In {\em Probabilistic methods for algorithmic discrete mathematics},
  volume~16 of {\em Algorithms Combin.}, pages 195--248. Springer, Berlin,
  1998.

\bibitem{MR4278608}
N.~Morrison and J.~A. Noel.
\newblock A sharp threshold for bootstrap percolation in a random hypergraph.
\newblock {\em Electron. J. Probab.}, 26:Paper No. 97, 85, 2021.

\bibitem{MR1438520}
A.~Panconesi and A.~Srinivasan.
\newblock Randomized distributed edge coloring via an extension of the
  {C}hernoff-{H}oeffding bounds.
\newblock {\em SIAM J. Comput.}, 26(2):350--368, 1997.

\bibitem{MR3958417}
G.~L. Torrisi, M.~Garetto, and E.~Leonardi.
\newblock A large deviation approach to super-critical bootstrap percolation on
  the random graph {$G_{n,p}$}.
\newblock {\em Stochastic Process. Appl.}, 129(6):1873--1902, 2019.

\bibitem{V07}
T.~Vallier.
\newblock {\em Random graph models and their applications}.
\newblock PhD thesis, Lund Univ., 2007.

\end{thebibliography}

\appendix
\section{Proof of Lemma~\ref{lem:negcorr}}\label{sec:conc}
	
		In order to prove the lemma we need the following results.
		
		\begin{lemma}\label{chtub}[\cite{MR1678578}, Lemma 2.4]
			For all $x\geq 0$,
			$$(1+x)\ln(1+x)-x\geq \frac{3x^2}{6+2x}.$$
		\end{lemma}
		
		\begin{lemma}\label{McDvar}[\cite{MR1678578}, Lemma 2.8]
			Let $g(x)=(e^x-1-x)/x^2$ if $x\neq 0$. Then the function $g$ is increasing and if the random variable $X$ satisfies $\mathbb{E}[X]=0$ and $X\leq b$, then 
			$$\mathbb{E}[e^X]\leq e^{g(b)\mathrm{Var}(X)}.$$
		\end{lemma}

	\begin{proof}[Proof of Lemma~\ref{lem:negcorr}]
		By Markov's inequality we have for any $h>0$
		$$\mathbb{P}[X\geq \mathbb{E}[\widehat{X}]+t]=\mathbb{P}\left[e^{hX}\geq e^{h(\mathbb{E}[\widehat{X}]+t)}\right]\\
		\leq \frac{\mathbb{E}[e^{hX}]e^{-h \mathbb{E}[\widehat{X}]}}{e^{ht}}.$$
		Due to the boundedness of $X$ we have
		$$\mathbb{E}[e^{hX}]=\mathbb{E}\left[\sum_{j=0}^{\infty} \frac{t^j X^j}{j!} \right]=\sum_{j=0}^\infty \frac{t^j \mathbb{E}[X^j]}{j!}.$$
		
		Now $X^j=(\sum_{i=0}^n X_i)^j$ is the sum of terms of the form $\prod_{i\in I} X_i$ for some $I\subseteq [n]$. Hence by linearity of expectation and our assumptions we have $\mathbb{E}[X^j]\leq \lambda\mathbb{E}[\widehat{X}]^j$ implying:
		$$\mathbb{E}[e^{hX}]\leq \lambda \sum_{j=0}^\infty \frac{t^j \mathbb{E}[\widehat{X}]^j}{j!}=\lambda\mathbb{E}[e^{h\widehat{X}}].$$
		
		Due to the independence of the $\widehat{X}_i$ we have:
		$$\mathbb{P}[X\geq \mathbb{E}[\widehat{X}]+t]\leq \lambda \frac{\mathbb{E}[e^{h\widehat{X}}]e^{-h \mathbb{E}[\widehat{X}]}}{e^{ht}}=\lambda \frac{\prod_{i=1}^n \mathbb{E}[e^{h(\widehat{X}_i-\mathbb{E}[\widehat{X}_i])}]}{e^{ht}}.$$
		Since $\widehat{X}_i$ is an indicator random variable, we have $h\widehat{X}_i\leq h$ and Lemma~\ref{McDvar} implies 
		$$\mathbb{E}[e^{h(\widehat{X}_i-\mathbb{E}[\widehat{X}_i])}]\leq e^{g(h)\mathrm{Var}(h\widehat{X}_i)}=e^{g(h)h^2\mathrm{Var}(\widehat{X}_i)}.$$
		
		Setting $h=\ln(1+t/\mathrm{Var}(\widehat{X}))$ gives us that:
		\begin{align*}
		\mathbb{P}[X\geq \mathbb{E}[\widehat{X}]+t]
		&\leq \lambda \exp\left(g(h)h^2\mathrm{Var}(\widehat{X})-h t \right)\\
		&=\lambda \exp\left(\left(\frac{t}{\mathrm{Var}(\widehat{X})}-h\right)\mathrm{Var}(\widehat{X})-h t \right)\\
		&=\lambda \exp\left(t-\ln\left(1+\frac{t}{\mathrm{Var}(\widehat{X})}\right)\mathrm{Var}(\widehat{X})- \ln\left(1+\frac{t}{\mathrm{Var}(\widehat{X})}\right) t \right)\\
		&=\lambda \exp\left(-\mathrm{Var}(\widehat{X})\left(\left(1+\frac{t}{\mathrm{Var}(\widehat{X})}\right)\ln\left(1+\frac{t}{\mathrm{Var}(\widehat{X})}\right)- \frac{t}{\mathrm{Var}(\widehat{X})} \right)\right),\\
		\end{align*}
	where we substitute $g(h)$ in the second step and all other occurrences of $h$ in the third step. 
		The result follows from Lemma~\ref{chtub}.
	\end{proof}
\end{document}